\documentclass[11pt,reqno]{amsart}
\usepackage[latin1]{inputenc}
\usepackage[T1]{fontenc}
\usepackage{amsbsy,amsfonts,amsmath,amssymb,amscd,amsthm}
\usepackage{mathrsfs,mathdots}
\usepackage{subfigure,enumitem,color,graphicx}
\usepackage[a4paper,text={6.1in,8in},centering,includefoot,foot=1in]{geometry}
\usepackage{pxfonts}
\usepackage[colorlinks=true,linkcolor=blue,citecolor=green]{hyperref}
\usepackage{srcltx}
\usepackage[belowskip=10pt,margin=20pt,font=small,labelfont=bf,textfont=it]{caption}

\small\normalsize
\usepackage{ifpdf}
\begingroup\makeatletter%
\ifx\SetFigFontNFSS\undefined\gdef\SetFigFontNFSS#1#2#3#4#5{%
\reset@font\fontsize{#1}{#2pt}\fontfamily{#3}\fontseries{#4}\fontshape{#5}\selectfont}\fi%
\endgroup%

\DeclareMathOperator{\bb}{\mathbb{R}}

\DeclareMathAlphabet{\mathpzc}{OT1}{pzc}{m}{it}

\theoremstyle{plain}
\newtheorem{thm}{Theorem}[section]
\newtheorem{mainthm}{Theorem}

\newtheorem{lem}[thm]{Lemma}
\newtheorem{prop}[thm]{Proposition}
\newtheorem{defi}[thm]{Definition}
\newtheorem{rem}[thm]{Remark}

\numberwithin{equation}{section}

\begin{document}
~\vspace{-0.1cm}
\title{Heteroclinic cycles arising in generic unfoldings of nilpotent
singularities}

\author[Barrientos]{Pablo G. Barrientos}
\address{\footnotesize \centerline{Departamento de Matem\'aticas, Universidad de Oviedo}
   \centerline{Av. Calvo Sotelo s/n, 33007 Oviedo, Spain}}
\email{barrientos@uniovi.es}

\author[Ib\'a\~nez]{Santiago Ib\'a\~nez}
\address{\footnotesize \centerline{Departamento de Matem\'aticas, Universidad de Oviedo}
   \centerline{Av. Calvo Sotelo s/n, 33007 Oviedo, Spain}}
\email{mesa@uniovi.es}

\author[Rodr\'iguez]{J. \'Angel Rodr\'iguez}
\address{\footnotesize \centerline{Departamento de Matem\'aticas, Universidad de Oviedo}
   \centerline{Av. Calvo Sotelo s/n, 33007 Oviedo, Spain}}
\email{jarodriguez@uniovi.es}


\begin{abstract}
In this paper we study the existence of heteroclinic cycles in
generic unfoldings of nilpotent singularities. Namely we prove
that any nilpotent singularity of codimension four in
$\mathbb{R}^4$ unfolds generically a bifurcation hypersurface of
bifocal homoclinic orbits, that is, homoclinic orbits to
equilibrium points with two pairs of complex eigenvalues. We also
prove that any nilpotent singularity of codimension three in
$\mathbb{R}^3$ unfolds generically a bifurcation curve of
heteroclinic cycles between two saddle-focus equilibrium points
with different stability indexes. Under generic assumptions these
cycles imply the existence of homoclinic bifurcations. Homoclinic
orbits to equilibrium points with complex eigenvalues are the
simplest configurations which can explain the existence of complex
dynamics as, for instance, strange attractors. The proof of the
arising of these dynamics from a singularity is a very useful
tool, particularly for applications.
\end{abstract}

\maketitle

\vspace{-1cm}
\section{Introduction}\label{section:intro} The relationship between dynamic complexity and the presence of homoclinic
orbits was discovered by Poincar\'{e} more than a century ago. In his famous essay on the stability of the solar system \cite{Poincare}, Poincar\'{e} showed that the invariant manifolds of a hyperbolic fixed point of a diffeomorphism could cut each other at points, called homoclinics, which yield the existence of more and more points of this type and consequently, a very complicated configuration of the manifolds. Many years later, Birkhoff \cite{Birkhoff35} showed that, in general, near a homoclinic point there exists an extremely intrincated set of periodic orbits, mostly with a very long period. By the mid 60's, Smale \cite{Smale} placed his geometrical device, the Smale horseshoe, in a neighborhood of a transversal homoclinic point. The horseshoes explained the Birkhoff's result and arranged the complicated dynamics that occur near a homoclinic orbit by means of a conjugation to the Bernouilli's shift. In \cite{MoraViana} authors proved the appearance of strange attractors during the process of creation or destruction of the Smale horseshoes which appear through a bifurcation of a tangential homoclinic point. These attractors are like those shown in \cite{BenedicksCarleson} for the H\'{e}non family, that is, they are nonhyperbolic and persistent in the sense of measure.

In the framework of vector fields, Shil'nikov \cite{Shilnikov65}
proved that in every neighborhood of a homoclinic orbit to a
hyperbolic equilibrium point of an analytical vector field on
$\mathbb{R}^{3}$, with eigenvalues $\lambda $ and $-\varrho
\pm\omega i$ such that $ 0<\varrho<\lambda $, that is, the
so-called \emph{Shil'nikov homoclinic orbit}, there exists a
countable set of periodic orbits. This result is similar to that
found by Birkhoff for diffeomorphisms and thus, it should be
understood in a manner similar to that devised  by Smale. Indeed,
Tresser \cite{Tresser} showed that in every neighborhood of such a
homoclinic  orbit,  an infinity of linked horseshoes can be
defined in such a way that the dynamics is conjugated to a
subshift of finite type on an infinite number of symbols. Once
again, these horseshoes appear and disappear by means of generic
homoclinic bifurcations leading to persistent non hyperbolic
strange attractors like those in \cite{MoraViana}.

As follows from \cite{OvsShi1987}, nonhyperbolic dynamics is dense
in the space $\mathcal{X}$ of vector fields with a Shil'nikov
homoclinic orbit. In particular, for each $\varepsilon>0$, the
subset of vector fields with a homoclinic tangency to a hyperbolic
periodic orbit in an $\varepsilon$-neighbourhood of the homoclinic
orbit is dense in $\mathcal{X}$. These tangencies give rise to
suspended H\'enon-like strange attractors. In \cite{Puma97,Puma01}
it was proved that infinitely many of these strange attractors can
coexist in non generic families of vector fields with a Shil'nikov
homoclinic orbit, for parameter values in a set of positive
Lebesgue measure. Later \cite{hom}, it was proved that an infinity
of such attractors can coexist in a more general context. For an
extensive study of the phenomena accompanying homoclinic
bifurcations, see \cite{bondiavia,homsan,PalisTakens}.

Because of the importance of homoclinic orbits in Dynamics, many
papers were devoted to prove their existence. A seminal work was
due to Melnikov \cite{Melnikov}, who introduced original ideas to
prove the existence of transversal homoclinic orbits in
non-autonomous perturbations of a planar hamiltonian vector field.
These ideas were developed in \cite{ChowHaleMalletParet80} in
order to determine both, homoclinic bifurcation curves and the
existence of subharmonics in two-parameter families of
non-autonomous second order differential equations. In
\cite{Palmer84}, Palmer developed a theory involving transversal
homoclinic points and exponential dichotomies that was very useful
for the study of homoclinic bifurcations in higher dimensions.

Since Shil'nikov homoclinic orbits are not transversal, Melnikov's
techniques had to be modified in order to prove their existence in
families of vector fields. In \cite{Rodriguez86}, generic families
of quadratic three dimensional vector fields with Shil'nikov
homoclinic orbits were given. Putting together ideas from
\cite{Rodriguez86,ChowHaleMalletParet80,Palmer84}, it was proved
in \cite{IbanezThesi} that Shil'nikov homoclinic orbits appear in
generic unfoldings of a nilpotent singularity of codimension four
in $\mathbb{R}^{3}$. Since singularities are the simplest elements
to be found in phase portraits of vector fields, arguing the
existence of homoclinic orbits from the presence of singularities
is a highly relevant task. Nevertheless, in order to get the
greatest interest in applications, such singularities should be of
codimension as low as possible. With this in mind, the result
obtained in \cite{IbanezThesi} was improved in
\cite{IbanezRodriguez}, where it was showed that
%\vspace{-0.5cm}
\begin{mainthm}
\label{theoremA}
Shil'nikov homoclinic orbits appear in every generic unfolding of the nilpotent singularity of codimension three in $\mathbb{R}^{3}$.
\end{mainthm}
%\vspace{-0.5cm}
Proving that Shil'nikov homoclinic orbits can be
unfolded generically from a singularity of codimension less than
three is currently a very interesting open problem. The dimension
of the center manifold should be at least three. The lowest
codimension singularities in $\mathbb{R}^{3}$ with a 3-dimensional
center manifold are the Hopf-zero singularities which have
codimension two \cite{GuckenheimerHolmes}. The difficulties that
appear on studying the existence of Shil'nikov homoclinic orbits
in generic unfoldings of Hopf-zero singularities are discussed in
\cite{DumortierIbanezKokubuSimo}.

Theorem~\ref{theoremA} was essential in \cite{Drubi} to prove the
existence of persistent strange attractors in the four parametric
family of vector fields obtained when two Brusselators are
linearly coupled by diffusion. Indeed, this family is a generic
unfolding of three-dimensional nilpotent singularities of
codimension three. Therefore it displays Shil'nikov homoclinic
orbits and, consequently, persistent strange attractors.
Nevertheless, this family may display a richer dynamics.
Three-dimensional nilpotent singularities appear along two
bifurcation curves which emerge from a bifurcation point
corresponding to a four-dimensional nilpotent singularity of
codimension four, for which the family is also a generic
unfolding. Therefore, one should wonder whether a different class
of homoclinic orbits can take place from this four-dimensional
nilpotent singularity. In this paper we will prove the following
result:

\begin{mainthm}
\label{thmB} In every generic unfolding of a four-dimensional
nilpotent singularity of codimension four there is a bifurcation
hypersurface of homoclinic orbits to equilibrium points with two
pairs of eigenvalues $\rho_{k}\pm\omega_{k} i$, with $k=1,2$, such
that $\rho_{1}<0<\rho_{2}$.
\end{mainthm}

Homoclinic orbits in Theorem~\ref{thmB} are usually known as
\emph{bifocal homoclinic orbits} or, shortly, \emph{bifocus}.
Shil'nikov \cite{Shilnikov67} was again the first one in studying
the dynamics associated to them. He proved, as in
\cite{Shilnikov65}, the existence of a countable set of periodic
orbits in the non-resonant case $-\rho_1\neq \rho_2$. Subsequent
works \cite{Devaney76,Sparrow91,glendinning97,harterich98} were
devoted to analyze the formation and bifurcations of these
periodic orbits by studying the Poincar\'{e} map associated to the
flow in a neighborhood of the bifocus. Devaney \cite{Devaney76}
considers the hamiltonian case, hence with $-\rho_1=\rho_2$. He
proves that for any local transverse section to the homoclinic
orbit, and for any positive integer $N$, there is a compact
inva\-riant hyperbolic set on which the Poincar\'e map is
conjugate to the Bernoulli shift on $N$ symbols. In seeking to
determine the invariant set of this Poincar\'{e} map in the
general case, it is shown in \cite{Sparrow91} that this set is
contained in a neighborhood of a spiral sheet (shaped like a
scroll). In fact, the invariant set is a neighborhood of the
intersection of this scroll and its image under the map, which is
another scroll, in general skewed and offset from the original. In
\cite{glendinning97} the authors extend the known theory regarding
bifocal homoclinic bifurcations and present numerical verification
of the more interesting theoretical predictions that had been
made. H\"{a}rterich \cite{harterich98} studies bifocal homoclinic
orbits arising in reversible systems, hence again with
$-\rho_1=\rho_2$. He proves that for any $N\geq 2$ there exists
infinitely many $N$-homoclinic orbits in a neighborhood of the
primary homoclinic orbit. Each of them is accumulated by one or
more families of $N$-periodic orbits.

As for Shil'nikov homoclinic orbits, it has been proved (see
\cite{OvsShi1992}) that homoclinic tangencies to hyperbolic
periodic orbits are dense in the space of vector fields with a
bifocal homoclinic orbit. Nevertheless, despite the abundant
literature regarding bifocus, as far as we know, no result has
been established relating the existence of bifocal homoclinic
bifurcations with the existence of persistent strange attractors.
This, in spite of a bifocus seems to be a scenario for more
complicated dynamics than those inherent to Shil'nikov homoclinic
orbits, where the existence of such strange attractors has been
proved. In fact it seems natural to think that the dynamical
complexity associated with homoclinic cycles increases with
dimension. For instance, strange attractors with more than one
positive Lyapunov exponent could appear. Therefore, bearing in
mind a possible extension of Theorem~\ref{thmB} to higher
dimensions, we will begin its proof working with $n$-dimensional
nilpotent singularities of codimension $n$.

Let $X$ be a $C^{\infty }$ vector field on $\mathbb{R}^{n}$ with
$X(0)=0$ and 1-jet linearly conjugated to
$\sum_{k=1}^{n-1}x_{k+1}\partial/\partial x_{k}$.
Introducing appropriate coordinates, $X$ can be written as%
\begin{equation*}
\sum_{k=1}^{n-1}x_{k+1}\frac{\partial }{\partial x_{k}}+f(x_{1},...,x_{n})%
\frac{\partial }{\partial x_{n}},
\end{equation*}%
with $f(x)=O(\left\Vert x\right\Vert ^{2})$ where $x=(x_{1},...,x_{n})$. It
is said that $X$ is a \emph{nilpotent singularity of codimension $n$} if the
generic condition $\partial ^{2}f/\partial x_{1}^{2}\neq 0$ is fulfilled.
As we will explain in Section 3, in appropriate coordinates and after
rescaling, any generic $n$-parametric unfolding $X_{\lambda }$ of a nilpotent
singularity can be written in a neighborhood of the origin as
\begin{equation}
\label{eq:intro1}
\sum_{k=1}^{n-1}y_{k+1}\frac{\partial }{\partial y_{k}}+(\nu
_{1}+\sum_{k=2}^{n}\nu _{k}y_{k}+y_{1}^{2}+O(\varepsilon ))\frac{\partial }{%
\partial y_{n}},
\end{equation}
with $\varepsilon >0$ and $\nu _{1}^{2}+...+\nu _{n}^{2}=1$. The limit family obtained for $\varepsilon =0$, will play a main role. It is time reversible with respect to the involution $R(y_{1},y_{2},...,y_{n})=((-1)^{n}y_{1},(-1)^{n-1}y_{2},...,-y_{n})$ for
parameter values on the set
\begin{equation*}
\mathcal{T}=\left\{ (\nu _{1},...,\nu _{n})\in \mathbb{S}^{n-1}:\nu _{n-2i}=0\text{ }
\mbox{with}\text{ }i=0,...,\lfloor(n-2)/2\rfloor\right\} ,
\end{equation*}%
where $\lfloor.\rfloor$ denotes the floor function. This manifold $\mathcal{T}$ of dimension
$\lfloor n/2-1\rfloor$ is called the reversibility set of the $n$-dimensional nilpotent
limit family.

For $n=4$ and for values of the parameters $\nu _{1}<0$, $\nu _{2}=\nu
_{4}=0 $, the limit family can be transformed in
\begin{equation}
\label{eq:intro2}
x_{2}\frac{\partial }{\partial x_{1}}+x_{3}\frac{\partial }{\partial x_{2}}%
+x_{4}\frac{\partial }{\partial x_{3}}+(-x_{1}+\eta _{3}x_{3}+x_{1}^{2})%
\frac{\partial }{\partial x_{4}},
\end{equation}
where $\eta _{3}=2^{-1/2}(-\nu _{1})^{-1/4}\nu _{3}$. Denoting $u=x_{1}$ and
$P=-\eta _{3}$ the vector field (\ref{eq:intro2}) is equivalent to the fourth-order
equation $u^{(iv)}(t)+Pu^{\prime \prime }(t)+u(t)-u^{2}(t)=0$, which has
been widely studied \cite{AmickToland92,ChampneysToland93,BuffoniChampneysToland96,Buffoni96} due to its role in some applications as the study of travelling waves of the Korteweg- de Vries equation
$$u_{t}=u_{xxxx}-bu_{xxx}+2uu_{x},$$
or the description of the displacement of a compressed strut with
bending softness resting on a nonlinear elastic foundation
\cite{Champneys98}. In particular, according to
\cite{AmickToland92}, when $\eta _{3}=2$ the vector field
(\ref{eq:intro2}) has a homoclinic orbit to a hyperbolic
equilibrium point at which the linear part has a pair of double
real eigenvalues $\pm 1$. We will complete the proof of
Theorem~\ref{thmB} by proving that the homoclinic  orbit persists
for parameter values on a hypersurface $\mathcal{H}om$ which
intersects $\varepsilon > 0$ and is obtained by studying the
appropriate bifurcation equation. Moreover we will show that
$\mathcal{H}om$ contains regions corresponding to bifocal
homoclinic orbits. An essential fact used in \cite{AmickToland92}
to prove the existence of homoclinic orbits in (\ref{eq:intro2})
is that it is a family of hamiltonian vector fields. This permits
to apply the general theory developed in \cite{HoferToland84}.
Again bearing in mind a possible extension of Theorem~\ref{thmB}
to higher dimensions we will prove that, for any even $n$ and for
parameter values on the reversibility set, the vector fields in
the limit family of (\ref{eq:intro1}) are hamiltonian.

Methods used in the proof of Theorem~\ref{thmB} also allow us to
prove the existence of topological Bykov cycles, which will be
defined below, in the case $n=3$, according with the following
result.
\begin{mainthm}
\label{thmC} In every generic unfolding of a three-dimensional
nilpotent singularity of codimension three there is a bifurcation
curve of topological Bykov cycles.
\end{mainthm}

For $n=3$ and when $\nu _{1}<0$ and $\nu _{2}<0$, the family (\ref{eq:intro1}) can be
transformed into the family
\begin{equation}
\label{eq:intro3}
x_{2}\frac{\partial }{\partial x_{1}}+x_{3}\frac{\partial }{\partial x_{2}}%
+(c^{2}-x_{2}+\nu x_{3}-x_{1}^{2}/2+O(\varepsilon ))\frac{\partial }{%
\partial x_{3}}
\end{equation}
where $c^{2}=2\nu _{1}/\nu _{2}^{3}$ and $\nu =\nu _{3}/(-1/\nu
_{2})^{1/2}$. When $\nu =0$ this limit family is equivalent to the
third-order equation \ $x^{\prime \prime \prime}(t)+x^{\prime
}(t)+x(t)^{2}/2=c^{2}$, which has been studied extensively in the
literature (see \cite{Kuramoto,lau,mic} and references therein)
because it plays a very relevant role in the study of the
existence of steady solutions and travelling waves of the
Kuramoto-Shivashinsky $u_{t}+u_{xxx}+u_{xx}+u_{x}^{2}/2=0$. For
each value of $c$, family (\ref{eq:intro3}) has two saddle-focus
equilibria $P_{\pm }$. In particular, the eigenvalues $\lambda$
and $-\rho\pm i\omega$ at $P_{-}$ satisfy that  $0<\rho<\lambda$ ,
that is, the spectral assumptions in Shil'nikov's theorem. In
\cite{Kuramoto} it is proved the existence of a heteroclinic
connection $\Gamma _{1}=W^{u}(P_{-})\cap W^{s}(P_{+})$ when
$c=c_{k}=15\sqrt{22/19^{3}}$. For this same value of $c$, it was
proved in \cite{IbanezRodriguez} the existence of a topologically
transverse intersection $\Gamma_{2}$ between the two-dimensional
invariant manifolds $W^{u}(P_{+})$ and $W^{s}(P_{-})$. Therefore,
$\Gamma =\Gamma _{1}\cup \Gamma_{2}\cup\{P_-,P_+\}$ is a
heteroclinic cycle, known in the literature as \emph{Bykov cycle}.
Bifurcations arising from the breaking of this codimension two
heteroclinic cycle have been widely studied in the literature
\cite{dumibakok2001,fernandez}. In particular, the birth of
homoclinic orbits was considered in \cite{dumibakok2001} for the
case which concerns us. It follows from these papers that in any
generic 2-parametric unfolding of a Bykov cycle there exist two
homoclinic bifurcation curves, which spiral in to the Bykov cycle
bifurcation point. Genericity can be read in terms of the
transversality between the 2-dimensional invariant manifolds,
which can be replaced by the condition of topological
transversality, and the generic splitting of the connection along
the 1-dimensional invariant manifolds (details can be seen in
\cite{dumibakok2001,fernandez,IbanezRodriguez}).

In this paper we will prove that family (\ref{eq:intro3}) exhibits
topological Bykov cycles for parameters along a curve
$\gamma=\{(c(\varepsilon),\nu(\varepsilon),\varepsilon)\, : \,
\varepsilon\in [0,\varepsilon_0)\}$ for some $\varepsilon_0>0$
with $c(0)=c_k$ and $\nu(0)=0$. Fixing a section $\Sigma$
transverse to $\Gamma_1$ we consider the splitting function
$h(c,\nu,\varepsilon)=(h_1(c,\nu,\varepsilon),h_2(c,\nu,\varepsilon))$
defined as the distance between $W^u(P_-(\tau))\cap\Sigma$ and
$W^s(P_+(\tau))\cap\Sigma$ for $\tau=(c,\nu,\varepsilon)$ close
enough to $\tau_k=(c_k,0,0)$. Note that $h$ take values on
$\mathbb{R}^2$ and existence of Bykov cycle is equivalent to
$h(c,\nu,\varepsilon)=(0,0)$. We will prove that
\begin{equation}
\label{intro:generic_condition}
\left|\begin{array}{ll}
\displaystyle{\frac{\partial h_1}{\partial c}(\tau_k)} & \displaystyle{\frac{\partial h_1}{\partial \nu}(\tau_k)}
\\
\\
\displaystyle{\frac{\partial h_2}{\partial c}(\tau_k)} & \displaystyle{\frac{\partial h_2}{\partial \nu}(\tau_k)}
\end{array}\right|\neq 0.
\end{equation}
Hence the existence of $\gamma$, and therefore Theorem~\ref{thmC},
follows from the Implicit Function Theorem.

We should remark that the above generic condition guarantees the
existence of Shil'nikov homoclinic orbits, in the sense of Theorem
A. Indeed, for each $\varepsilon>0$ small enough and fixed,
(\ref{intro:generic_condition}) implies that the splitting
function is a local diffeomorphism or, in other words, that the
splitting of the connection along the 1-dimensional invariant
manifolds is generic. Therefore there exists a Shil'nikov
bifurcation surface shaped as a scroll around $\gamma$. In the
proof of Theorem~\ref{theoremA} given in \cite{IbanezRodriguez},
for the sake of brevity, we did not include the computation of the
above generic condition, although we appealed to it. Since the
publication of that paper such computation has been frequently
demanded to us.

The paper is organized as follows. In \S\ref{sec:dico} we include
a brief summary of results about dichotomies in order to get a
precise formulation of the bifurcation equations which are
required in the subsequent sections to prove the existence of
heteroclinic and homoclinic orbits. \S\ref{section:nilpotent-n} is
devoted to introduce nilpotent singularities on $\mathbb{R}^{n}$
and the limit families obtained after rescaling properly a generic
unfolding. There we also state that when $n$ is an even number and
the parameter values belong to the reversibility set, the limit
family consists of hamiltonian vector fields. Theorem~\ref{thmB}
and Theorem~\ref{thmC} are proved in \S\ref{section:nilpotent-4}
and \S\ref{section:nilpotent-3}, respectively.

\vspace{-0.6cm}
\section{Dichotomies and bifurcation equations.} \label{sec:dico}
Let $x^{\prime}=f(x)$ be a nonlinear equation, where $x\in \mathbb{R}^{n}$ and $f$ is a regular enough vector field, and assume that it has a heteroclinic orbit $\gamma =\left\{ p(t):t\in \mathbb{R}\right\} $ connecting two hyperbolic equilibrium points $p_{+} $ and $p_{-}$ (if $p_{+}=p_{-}$, $\gamma$ is said homoclinic). Consider a family
\begin{equation}
\label{*1}
x^{\prime }=f(x)+g(\lambda,x),
\end{equation}
with $\lambda \in \mathbb{R}^{k}$ and $g$ regular enough, such that $g(0,x)=0$. For any $\lambda $ small enough, family (\ref{*1}) has hyperbolic equilibrium points $p_{+}(\lambda )$ and $p_{-}(\lambda )$, conti\-nuation of $p_{+}$ and $p_{-}$, respectively, and the stability index is preserved. In order to study the persistence of the heteroclinic orbit for $\lambda$ small enough we introduce the change of variables $x(t)=z(t)+p(t)$ in (\ref{*1}) to obtain
\begin{equation}
\label{*2}
z^{\prime }(t)=Df(p(t))z(t)+b(\lambda ,t,z(t)),
\end{equation}
where
\begin{equation*}
b(\lambda ,t,z(t))=f(p(t)+z(t))-f(p(t))-Df(p(t))z(t)+g(\lambda,p(t)+z(t)).
\end{equation*}%
Notice that $b(0,t,0)=D_{z}b(0,t,0)=0$ for all $t\in \mathbb{R}$.

Persistence of heteroclinic orbits in (\ref{*1}) implies the existence of bounded solutions for (\ref{*2}) which, in turn, implies the existence of bounded solutions for a equation as
\begin{equation}
\label{*3}
z^{\prime }(t)=Df(p(t))z(t)+b(t),
\end{equation}%
where $b$ belongs to the space
$C_{b}^{0}(\mathbb{R},\mathbb{R}^{n})$. In the sequel
$C_{b}^{k}(\mathbb{R},\mathbb{R}^{n})$ denotes the Banach space of
bounded continuous $\mathbb{R}^n$-valued functions whose
derivatives up to order $k$ exist and are bounded and continuous.
The existence of bounded solutions of a linear equation
$x^{\prime}=A(t)x+b(t)$, as that in (\ref{*3}), will be given in
terms of exponential dichotomies of the homogeneous equation
$x^{\prime}=A(t)x$ and its adjoint $w^{\prime}=-A(t)^{\ast}w$,
where $A(t)^{\ast}$ denotes the conjugate transpose of $A(t)$. The
classical references for the study of exponential dichotomies are
\cite{Massera-Schaffer,Coppel,Palmer84,Palmer00}.

\subsection{Exponential dichotomy}

Let $X(t)$ be a fundamental matrix of
\begin{equation}
\label{eq:2.1}
x^{\prime }=A(t)x,\qquad x\in \mathbb{R}^{n},
\end{equation}
where $A(t)$ is defined and continuous on an interval $J\subseteq \mathbb{R}$.

\begin{defi}
\label{def:dicotomia} It is said that the equation (\ref{eq:2.1})
has an exponential dichotomy on $J$ if there exists a projection
$P:\mathbb{R}^n\to\mathbb{R}^n$, that is, an $n$ by $n$ matrix $P$
with $P^2=P$, and positive constants $K$, $L$, $\alpha$ and
$\beta$ such that for every $s,t\in J$,
\begin{equation}
\label{eq:2.2}
\begin{alignedat}[c]{2}
\|X(t)PX^{-1}(s)\|&\leq K e^{-\alpha(t-s)} \quad && \text{for} \ t\geq s,
\\ \|X(t)(I-P)X^{-1}(s)\|&\leq L e^{-\beta(s-t)}\quad && \text{for} \ s\geq t.
\end{alignedat}
\end{equation}

\end{defi}

Let us define $\mathscr{P}(s)=X(s)PX^{-1}(s)$ for each $s\in J$. Notice that, according with the above definition, $\mathscr{P}(s)$ is the projection corresponding to the fundamental matrix $Y(t)=X(t)X^{-1}(s)$ of (\ref{eq:2.1}) and we can give an alternative definition of exponential dichotomy.
\begin{defi}
\label{eq:2.3}
It is said that the equation (\ref{eq:2.1}) has an exponential dichotomy on $J$ if for all $s\in J$ there exists a projection $\mathscr{P}(s):\mathbb{R}^n \to \mathbb{R}^n$ and positive constants $K$, $L$, $\alpha$ and $\beta$ independents of $s$ such that for all $t\in J$ the matrix $X^{-1}(t)\mathscr{P}(t)X(t)$ has constant coefficients and
\begin{equation*}
\begin{alignedat}[c]{2}
\|X(t)X^{-1}(s)\mathscr{P}(s)\|&\leq K e^{-\alpha(t-s)} \quad & &\text{for all} \ t\geq s, \\
\|X(t)X^{-1}(s)(I-\mathscr{P}(s))\|&\leq L e^{-\beta(s-t)} \quad & &\text{for all} \ s\geq t.
\end{alignedat}
\end{equation*}
\end{defi}

Although the notion of exponential dichotomy is stated for any $J\subseteq\mathbb{R}$, the most interesting cases are when $J$ is not bounded. We are particularly interested in $J=[\tau,\infty)$ or $J=(-\infty,\tau]$. In such cases the notions of stable and unstable subspaces can be introduced in terms of the ranges of the projections of the exponential dichotomies.

\begin{defi}
\label{def:Es} Suppose that the matrix $A(t)$ in \eqref{eq:2.1} is
defined and continuous on $J=[\tau,\infty)$ (resp.
$J=(-\infty,\tau]$). For each $t_0\in J$ the stable (resp.
unstable) subspace for initial time $t=t_0$ is defined as
\begin{align*}
E^{s}_{t_0}&=\{\xi \in \mathbb{R}^n: \|X(t)X^{-1}(t_0)\xi\|\to 0 \ \mathrm{when} \ t\to \infty \} \\
(\text{resp. } E^{u}_{t_0}&=\{\xi\in \mathbb{R}^n: \|X(t)X^{-1}(t_0)\xi\|\to 0 \ \mathrm{when} \ t\to -\infty \}).
\end{align*}
\end{defi}

Below we give a collection of results which can be helpful to follow the paper. Their proofs are available in the literature.

\begin{prop}
\label{lem:1}
Suppose that the equation $x^{\prime}=A(t)x$ has an exponential dichotomy on $J$.
\begin{enumerate}
\item[i)] When $J=[\tau,\infty)$, $E^s_{t_0}$ coincides with the range $\mathcal{R}(\mathscr{P}(t_0))$ of $\mathscr{P}(t_0)$ for all $t_0\in J$. Furthermore
\begin{equation*}
\mathcal{R}(\mathscr{P}(t_0))= \{\xi\in\mathbb{R}^n: \ \sup_{t\geq t_0}\|X(t)X^{-1}(t_0) \xi\| < \infty \},
\end{equation*}
and for all $t_0, t_1 \in J$ it follows that $E^s_{t_1}=X(t_1)X^{-1}(t_0)E^s_{t_0}$.

\item[ii)] When $J=(-\infty,\tau]$, $E^u_{t_0}$ coincides with the kernel $\mathcal{N}(\mathscr{P}(t_0))$ of $\mathscr{P}(t_0)$ for all $t_0\in J$. Furthermore
\begin{equation*}
\mathcal{N}(\mathscr{P}(t_0))= \{\xi\in\mathbb{R}^n: \ \sup_{t\leq t_0}\|X(t)X^{-1}(t_0) \xi\| < \infty \},
\end{equation*}
and for all $t_0, t_1 \in J$ it follows that $E^u_{t_1}=X(t_1)X^{-1}(t_0)E^u_{t_0}$.
\end{enumerate}
\end{prop}

From the above proposition it follows that the linear flow sends $E_{t_{0}}^{s}$ and $E_{t_{0}}^{u}$ to
$E_{t_{1}}^{s}$ and $E_{t_{1}}^{u}$, respectively. Accordingly, once $E_{t_{0}}^{s}$ and $E_{t_{0}}^{u}$ are fixed, the stable and unstable subspaces are determined for all $t$. Therefore, the projections are also determined for each $t\in J$ once they are defined for $t=t_{0}$. The same observation follows taking into account the uniqueness of solutions for the equation
\begin{equation*}
\mathscr{P}^{\prime }(s)=X^{\prime }(s)PX^{-1}(s)+X(s)P(X^{-1}(s))^{\prime
}=A(s)\mathscr{P}(s)-\mathscr{P}(s)A(s).
\end{equation*}

\begin{lem}
\label{lem:Coppel} If the linear homogeneous equation $x^{\prime}=A(t)x$, with $t\in (-\infty,\infty)$, has exponential dichotomy $[\tau, \infty)$ (resp. $(-\infty, \tau]$) for some $\tau \in \mathbb{R}$ then it has exponential dichotomy on $[t_0,\infty)$ (resp. $(-\infty,t_0]$) for all $t_0\in\mathbb{R}$.
\end{lem}

The next result \cite[Lemma 7.4]{Palmer00} states that exponential dichotomy is a robust property with respect to small enough perturbations of $A(t)$.

\begin{prop}
\label{thm:roughness} Suppose that $x^{\prime}=A(t)x $ has an
exponential dichotomy on $J=[a,b]$ (with $-\infty \leq a <b\leq
\infty$) with projection matrix function $\mathscr{P}(t)$, with
constants $K_1$, $K_2$ and exponents $\alpha_1$, $\alpha_2$. Let
$\beta_1$ and $\beta_2$ be such that $0<\beta_1<\alpha_1$ and
$0<\beta_2<\alpha_2$. Then there exists
$\delta_0=\delta_0(K_1,K_2,\alpha_1,\alpha_2,\beta_1,\beta_2)>0$
such that if $B(t)$ is a continuous matrix function with
$$
\|B(t)\|\leq\delta_t \leq \delta_0 \quad \text{for all $t \in J$,}
$$
the perturbed system
\begin{equation*}
x^{\prime}=[A(t)+B(t)]x
\end{equation*}
has an exponential dichotomy on $J$ with constants $L_1$, $L_2$
exponents $\beta_1$, $\beta_2$ and projection matrix
$\mathscr{Q}(t)$ satisfying that
$$\|\mathscr{Q}(t)-\mathscr{P}(t)\| \leq N \delta_t,
$$ where $L_1$,
$L_2$ and $N$ are constants which only depend on $K_1$, $K_2$,
$\alpha_1$ and $\alpha_2 $.
\end{prop}

From the above result and Lemma \ref{lem:Coppel} it  follows the
existence of an exponential dichotomy for the homogeneous part
$z^{\prime }=Df(p(t))z$ of the equation~(\ref{*3}). Since
$$
\lim_{t\rightarrow \infty }p(t)=p_{+} \quad \text{and} \quad
\lim_{t\rightarrow -\infty }p(t)=p_{-}
$$
and according to
Proposition~\ref{thm:roughness}, the equation $x^{\prime
}=Df(p(t))x$ has the same exponential dichotomy than  $x^{\prime
}=Df(p_{+})x$ (resp. $x^{\prime }=Df(p_{-})x$) on $[t_{0},\infty
)$ (resp. $(-\infty ,t_{0}]$). That is, if the stable (resp.
unstable) subspace of $x^{\prime}=Df(p_{+})x$ (resp. $x^{\prime
}=Df(p_{-})x$) has dimension $k$ then $ x^{\prime }=Df(p(t))x$ has
an exponential dichotomy on $[t_{0},\infty )$ (resp. $(-\infty
,t_{0}]$) with stable subspace $E_{t_{0}}^{s}$ (resp. unstable
subspace $E_{t_{0}}^{u}$) with dimension $k$. In fact we have the
following result:

\begin{prop}
\label{lem:conexion} Let $p(t)$ be a solution of the equation
$x^{\prime}=f(x)$ parametrizing an orbit on the stable (resp.
unstable) manifold of an equilibrium point $p$. Hence the
variational equation $x^{\prime }=Df(p(t))x$ has exponential
dichotomy on $[t_{0},\infty )$ (resp. $(-\infty,t_{0}]$).
Moreover,
$$
\mathcal{R}(\mathscr{P}(t_{0}))=T_{p(t_{0})}W^{s}(p) \quad
\text{(resp.
$\mathcal{N}(\mathscr{P}(t_{0}))=T_{p(t_{0})}W^{u}(p)$).}
$$
\end{prop}

\pagebreak  Now we can apply to (\ref{*3}) the result below, which
relates the existence of bounded solutions for a linear equation
and for its adjoint.

\begin{thm}{\rm \cite[Lemma 4.2]{Palmer84}}
\label{thm:Fredholm} Let $A(t)$ be a bounded and continuous matrix
defined on $(-\infty,\infty)$. The linear equation
$x^{\prime}=A(t)x$ has exponential dichotomy on $[t_0,\infty)$ and
on $(-\infty,t_0]$ if and only if the linear operator
$$L: x(t)
\in C^1_b(\mathbb{R},\mathbb{R}^n) \mapsto x^{\prime}(t)-A(t)x(t)
\in C^0_b(\mathbb{R},\mathbb{R}^n)$$ is Fredholm. The index of $L$
is $\dim E^s_{t_0} + \dim E^u_{t_0}-n$. Moreover,
$b\in\mathcal{R}(L)$ if and only if
\begin{equation*}
\int_{-\infty}^{\infty} <w(t), b(t)> \, dt =0
\end{equation*}
for all bounded solutions $w(t)$ of the adjoint equation $w^{\prime}=-A(t)^*w $.
\end{thm}

To explore the existence of bounded solutions of the adjoint equation one has to study its properties of exponential dichotomy.

\subsection{Exponential dichotomy for the adjoint equation}
\label{subsec:expon}

Let $X(t)$ be a fundamental matrix of the equation $x^{\prime }=A(t)x$. It is well known that the conjugate transpose of its inverse $X^{-1}(t)^{\ast}$ is a fundamental matrix of the adjoint equation $w^{\prime }=-A(t)^{\ast }w$. From this relationship between the fundamental matrices of both equations we can conclude the following result about the connection between their respective dichotomies.

\begin{prop}
\label{lem:adjunta} If the linear equation $x^{\prime}=A(t)x$ has exponential dichotomy on $J$ with projection matrix  $\mathscr{P}(t)$ then the adjoint equation $w^{\prime}=-A(t)^*w$ has exponential dichotomy on $J$ with projection matrix $I-\mathscr{P}(t)^*$. Moreover, for each $t_0 \in J$
\begin{align*}
\mathbb{R}^n&=\mathcal{R}(\mathscr{P}(t_0)) \, \bot \, \mathcal{R}(I-%
\mathscr{P}(t_0)^*) = \mathcal{R}(\mathscr{P}(t_0)) \, \bot \ \mathcal{N}(%
\mathscr{P}(t_0)^*), \\
\mathbb{R}^n &=\mathcal{R}(I-\mathscr{P}(t_0)) \, \bot \ \mathcal{R}(%
\mathscr{P}(t_0)^*)=\mathcal{N}(\mathscr{P}(t_0)) \, \bot \ \mathcal{R}(%
\mathscr{P}(t_0)^*).
\end{align*}
\end{prop}

As done in Definition~\ref{def:Es} we can define now the stable and unstable subspaces for adjoint equations.

\begin{defi}
Suppose that $J=[\tau,\infty)$ (resp. $J=(-\infty,\tau]$) is contained in the interval of definition of $x^{\prime}=A(t)x$. For each $t_0\in J$ the stable (resp. unstable) subspace for initial time $t=t_0$ of the adjoint equation $x^{\prime}=-A(t)^*x$ is defined as
\begin{align*}
E^{s*}_{t_0} &=\{w\in\mathbb{R}^n: \|X^{-1}(t)^*X(t_0)^*w\| \to 0 \ \mathrm{when} \ t\to \infty \}
\\
(\text{resp.\ } E^{u*}_{t_0}&=\{w\in\mathbb{R}^n: \|X^{-1}(t)^*X(t_0)^*w\|\to 0 \ \mathrm{when} \ t\to -\infty \}).
\end{align*}
\end{defi}

The following result about the relationship  between the invariant
subspaces of the equation $x^{\prime}=A(t)x$ and its adjoint
follows as a straight consequence of Proposition~\ref{lem:1} and
Proposition~\ref{lem:adjunta}. \pagebreak

\begin{prop}
\label{not:2} Suppose that the equation $x^{\prime}=A(t)x$ with $x\in\mathbb{R}^n$ and $t\in J$ has exponential dichotomy in $J$.
\begin{enumerate}
\item If $J=[t_0,\infty)$ then
\begin{align*}
E^s_{t_0}&=\mathcal{R}(\mathscr{P}(t_0))= \{x\in\mathbb{R}^n: \ \sup_{t\geq t_0
}\|X(t)X^{-1}(t_0)x\| < \infty \}, \\
E^{s*}_{t_0}&=\mathcal{N}(\mathscr{P}(t_0)^*)= \{w\in\mathbb{R}^n: \ \sup_{t\geq t_0
}\|X^{-1}(t)^*X(t_0)^* w\| < \infty \},
\end{align*}
and $\mathbb{R}^n=E^s_{t_0} \, \bot \, E^{s*}_{t_0}$. \\

\item If $J=(-\infty,t_0]$ then
\begin{align*}
E^u_{t_0}&=\mathcal{N}(\mathscr{P}(t_0))= \{x\in\mathbb{R}^n: \ \sup_{t\leq t_0
}\|X(t)X^{-1}(t_0)x\| < \infty \}, \\
E^{u*}_{t_0}&=\mathcal{R}(\mathscr{P}(t_0)^*)= \{w\in\mathbb{R}^n: \ \sup_{t\leq t_0
}\|X^{-1}(t)^*X(t_0)^* w\| < \infty \},
\end{align*}
and $\mathbb{R}^n=E^u_{t_0} \, \bot \, E^{u*}_{t_0}$.
\end{enumerate}
\end{prop}

In short, if the linear equation $x^{\prime}=A(t)x$ has exponential dichotomy in $J=[t_0, \infty)$ (resp. $(-\infty, t_0]$)
then the forward (resp. backward) bounded solutions of this equation and its adjoint are those which tend to zero exponentially when  $t\to \infty$ (resp. $t\to -\infty$). On the other hand, from the decompositions of $\mathbb{R}^n$ given in Proposition~\ref{not:2} it follows that, if $x^{\prime}=A(t)x$ has $m$ linearly independent forward (resp. backward) bounded solutions, then the adjoint equation $w^{\prime}=-A(t)^*w$ has $n-m$ linearly independent forward (resp. backward) bounded solutions.

\begin{prop}
\label{lem:formula} If the linear equation $x^{\prime}=A(t)x$ has exponential dichotomy in $[t_0,\infty)$ and in $(-\infty,t_0]$ then the number of linearly independent bounded solutions of the adjoint equation $w^{\prime}=-A(t)^*w$ is
\begin{equation*}
\dim E^{s*}_{t_0} \cap E^{u*}_{t_0} =n-\dim E^s_{t_0} -\dim E^u_{t_0} + \dim E^s_{t_0} \cap E^u_{t_0}.
\end{equation*}
\end{prop}

Now we apply the above result to determine the number of bounded solutions of the adjoint equation $z^{\prime}=-Df(p(t))^{\ast}z$. As we have already noticed, the number of linearly independent forward (resp. backward) bounded solutions of the variational equation $x^{\prime }=Df(p(t))x$ is given by the dimension of the stable (resp. unstable) subspace of the equation $x^{\prime }=Df(p_{+})x$ (resp. $x^{\prime}=Df(p_{-})x$). That is, such number coincides with the dimension of $W^{s}(p_{+})$ (resp. $W^{u}(p_{-})$). Therefore, taking into account that $E_{t_{0}}^{s}=T_{p(t_{0})}W^{s}(p_{+})$ and $E_{t_{0}}^{u}=T_{p(t_{0})}W^{u}(p_{-})$, we can conclude, from Proposition~\ref{lem:formula}, the following result.

\begin{prop}
\label{not:3} If $p(t)$ is a (homo)heteroclinic solution connecting two equilibrium points $p_{+}$ and $p_{-}$ then the number of linearly independent bounded solutions of the adjoint variational equation $w^{\prime}=-Df(p(t))^{\ast}w$ is the codimension of $T_{p(t_{0})}W^{s}(p_{+})+T_{p(t_{0})}W^{u}(p_{-})$, that is,
\begin{equation*}
n-\dim W^{s}(p_{+})-\dim W^{u}(p_{-})+\dim T_{p(t_{0})}W^{s}(p_{+})\cap T_{p(t_{0})}W^{u}(p_{-}).
\end{equation*}
\end{prop}

\begin{defi}
\label{def:nondegeneratehomoclinicorbit}
A (homo)heteroclinic orbit $\gamma $ is said non degenerate if
$$\dim T_{p}W^{s}(p_{+})\cap T_{p}W^{u}(p_{-})=1,$$
with $p\in\gamma$. Otherwise $\gamma$ is said degenerate.
\end{defi}

\begin{rem}
\label{rem:d} If the (homo)heteroclinic orbit is non degenerate,
the number of linearly independent bounded solutions is obtained
directly from the stability indexes of $p_{+}$ and $p_{-}$.
Moreover, although $\dim T_{p(t_{0})}W^{s}(p_{+})=\dim
W^{s}(p_{+})$ and $\dim T_{p(t_{0})}W^{u}(p_{-})=\dim
W^{u}(p_{-})$, in general $\dim T_{p(t_{0})}W^{s}(p_{+})\cap
T_{p(t_{0})}W^{u}(p_{-})$ does not coincide with $\dim
W^{s}(p_{+})\cap W^{u}(p_{-})$.
\end{rem}

In the sequel the (homo)heteroclinic orbit $\gamma=\{p(t)\,:\,t\in\mathbb{R}\}$ will be non degenerate.

%\vspace{0.2cm}
\subsection{Bifurcation equation}
As already mentioned, the existence of  (homo)heteroclinic orbits
for (\ref{*1}) implies the existence of bounded solutions of
(\ref{*2}) and, consequently, the existence of bounded solutions
of (\ref{*3}) when $b(t)\in C_{b}^{0}(\mathbb{R},\mathbb{R}^{n})$.
According to Proposition~\ref{thm:Fredholm}, if the adjoint
variational equation $w^{\prime}=-Df(p(t))^{\ast}w$ has $d$
linearly independent bounded solutions $w_{i}$, then the
persistence of the (homo)heteroclinic orbit requires the
fulfillment of the $d$ conditions
$$
\int_{-\infty }^{\infty }\left\langle w_{i}(t),b(t)\right\rangle\,
dt=0 \quad \text{for $i=1,\dots,d$}.
$$
The question now is the
sufficiency of such conditions.

When $d=1$ the sufficiency could be followed from
\cite{ChowHaleMalletParet80}. In general, for $d\geq 1$, the
techniques to be used follow the first steps of the Lin's method
\cite{Lin, Sandstede}. For $\|\lambda \|$ small enough, one has to
look for solutions $p_{\lambda }^{+}(\cdot )$ and $p_{\lambda
}^{-}(\cdot )$ of~(\ref{*1}), contained in the stable and unstable
invariant manifolds of the equilibrium points $p_{+}(\lambda )$
and $p_{-}(\lambda )$, respectively (see Figure~\ref{fig:B}).
Initial values $p_{\lambda }^{\pm }(t_0)$ will belong to a section
$\Sigma _{t_{0}}$ transverse to the (homo)heteroclinic orbit
$\gamma$. Namely
$$
\Sigma_{t_{0}}=p(t_0) + \{f(p(t_0))\}^{\bot}=p(t_0) + \left(W_{t_0}^+ \oplus W_{t_0}^- \oplus E_{t_{0}}^{\ast }\right)
$$
where $E_{t_{0}}^{\ast }=E_{t_{0}}^{s\ast }\cap E_{t_{0}}^{u\ast
}$  and $W_{t_0}^+$ (resp. $W_{t_0}^-$) is the orthogonal
complement of
$$
\text{$E_{t_0}^s \cap E_{t_0}^u=\mathrm{span}\{f(p(t_0))\}$ in
$E_{t_0}^s$ (resp. $E_{t_0}^u$).}
$$
Moreover the condition $\xi
^{\infty }(\lambda )=p_{\lambda}^{-}(t_{0})-p_{\lambda
}^{+}(t_{0})\in E_{t_{0}}^{\ast }$ will be required. Under these
assumptions there will exist two unique solutions $p_{\lambda
}^{\pm }(\cdot )$ for each $\lambda$. The jump
$$
\xi ^{\infty
}(\lambda )=p_{\lambda}^{-}(t_{0})-p_{\lambda }^{+}(t_{0})
$$
measures the displacement between the stable and unstable
invariant manifolds on the section $\Sigma _{t_{0}}$ in the
direction of the subspace $E_{t_{0}}^{\ast
}=[E_{t_{0}}^{s}+E_{t_{0}}^{u}]^{\bot}$.

The proof of the result below can be found in~\cite[Lemma
3.3]{Sandstede} and~\cite[Lemma 2.1.2]{Knobloch}. Namely, in
\cite{Knobloch} only the first item is proved and, moreover, the
proof is developed for the degenerate case although the non
degenerate one follows in a similar manner. The second item is
proved in \cite{Sandstede} for the non degenerate case. We include
in Appendix~\ref{apendixC} a complete and simplified proof of this
result. \pagebreak

\begin{figure}
\begin{center}
\resizebox{\textwidth}{!}{\input{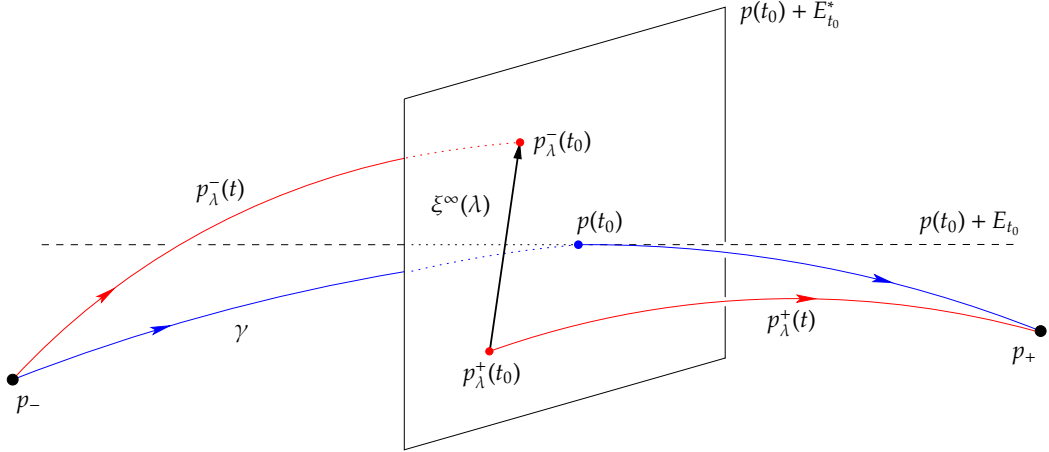}}
\caption{Non-degenerate heteroclinic orbit in $\bb^3$ where the
1-dimensional manifolds coincide. In this case,
$E^s_{t_0}=E^u_{t_0}=E_{t_0}$ (unidimensional),
$E^{s*}_{t_0}=E^{u*}_{t_0}=E^*_{t_0}$ (bidimensional) and
$\Sigma_{t_0}=p(t_0)+E^*_{t_0}$. For simplicity, we have assume
that the perturbation satisfies $g(\lambda,p_\pm)=0$ for all
$\lambda$.} \label{fig:B}
\end{center}
\end{figure}

\begin{lem}
\label{lem:eq-bif}
There exists $\delta>0$ such that for all $\lambda\in\mathbb{R}^k$, with $\|\lambda\|<\delta$,
\begin{enumerate}
\item There exists a unique pair of solutions $p_{\lambda }^{+}(t)$ and $p_{\lambda}^{-}(t)$ of~(\ref{*1}) parametrizing orbits on
$W^{s}(p_{+}(\lambda ))$ and $W^{u}(p_{-}(\lambda ))$, respectively, such that $p_{\lambda }^{\pm }(t_{0})\in \Sigma _{t_{0}}$
and
$$
\xi ^{\infty }(\lambda )=p_{\lambda }^{-}(t_{0})-p_{\lambda }^{+}(t_{0})\in
E_{t_{0}}^{\ast }.
$$
Writing the solutions as $p_{\lambda }^{\pm}(t)=p(t)+z_{\lambda }^{\pm }(t)$, then $z_{\lambda }^{\pm }(\cdot )$
are, respectively, forward and backward bounded solutions of the equation~(\ref{*2}). They depend regularly on $\lambda $ and the functions
$z_{0}^{\pm }$ are identically zero.

\item For $\varepsilon >0$ small enough, there exists a (homo)heteroclinic
solution $p_{\lambda }(t)$ such that
$$
\text{$\|p_{\lambda}(t_{0})-p(t_{0})\|<\varepsilon$ if and only if
$\xi ^{\infty}(\lambda )=0$,}
$$
that is, the components $\xi _{i}^{\infty }(\lambda )$ of
$\xi^{\infty }(\lambda )$ in a basis $\{w_{i}:i=1\dots d\}$ of
$E_{t_{0}}^{\ast }$ satisfy
\begin{equation*}
\qquad \xi _{i}^{\infty }(\lambda )  \equiv
\int_{-\infty}^{t_{0}}<w_{i}(s),b(\lambda,s,z_{\lambda
}^{-}(s))>\,ds  +\int_{t_{0}}^{\infty
}<w_{i}(s),b(\lambda,s,z_{\lambda }^{+}(s))>\,ds  =  0
\end{equation*}
being $w_i(s)=X^{-1}(s)^*X(t_0)^*w_i$ for $i=1,\dots,d$ bounded
linearly independent solutions of the adjoint variational
equation.
\end{enumerate}
\end{lem}

According with the above statement the  persistence of
(homo)heteroclinic orbits follows from the analysis of the
bifurcation equation $\xi^{\infty}(\lambda)=0$. The existence of
non zero parameter values $\lambda \in \mathbb{R}^{k}$ such that
$\xi ^{\infty}(\lambda )=0$ follows from the Implicit Function
Theorem when $D_{\lambda }\xi ^{\infty }(0)$ has rank $d<k$. Thus,
the following result follows: \pagebreak

\begin{thm}
\label{thm:bif} Let $\xi ^{\infty }(\lambda )=0$, with $\lambda \in \mathbb{R}^{k}$, be the bifurcation equation of the differential equation~(\ref{*1}). If $k>d$ and $\mathrm{rank}\,D_{\lambda }\xi ^{\infty }(0)=d$, then~(\ref{*1}) has a (homo)heteroclinic orbit for each parameter value $\lambda $ on a regular manifold of dimension $k-d$ with tangent subspace at $\lambda =0$ given by the solutions of the system
$$
\sum_{j=1}^{k}\xi _{ij}^{\infty }\lambda _{j}=0\qquad i=1,\dots ,d
$$
where
$$
\xi _{ij}^{\infty }\equiv \frac{\partial\xi _{i}^{\infty
}}{\partial\lambda _{j}}(0)=\int_{-\infty }^{\infty }<w_{i}(s),D_{\lambda _{j}}g(0,p(s))>\,ds
$$
for $i=1,\dots ,d$ and $j=1,\dots ,k$.
\end{thm}

\begin{rem}
Note that, when $k\leq d$, $\lambda =0$ is the unique value of $\lambda \in\mathbb{R}^{k}$ for which there exists a (homo)heteroclinic orbit $\gamma _{\lambda }=\{p_{\lambda }(t):\ p_{\lambda }^{\prime}(t)=f(p_{\lambda }(t))+g(\lambda,p_{\lambda }(t))\ t\in \mathbb{R}\}$ such that $\sup_{t\in \mathbb{R}}\|p_{\lambda }(t)-p(t)\|$ is small enough. If $k>d$ the homoclinic connection persists for parameter values on a manifold of codimension $d$ where
$$
d=n-\dim W^{s}(p_{+})-\dim W^{u}(p_{-})+1.
$$
In such a case we say that there is (homo)heteroclinic bifurcation of a non degenerate orbit at $\lambda=0$ which is of codimension $d$.
\end{rem}

\section{Nilpotent singularities of codimension $n$ on $\mathbb{R}^n$}\label{section:nilpotent-n}
\subsection{Generic unfoldings}
Let $X$ be a $C^\infty$ vector field in $\mathbb{R}^n$ with $X(0)=0$ and 1-jet at the origin linearly conjugated to $ \sum_{k=1}^{n-1} x_{k+1} \partial/\partial x_k$. Introducing appropriate $C^\infty$ coordinates, $X$ can be written as:
\begin{equation} \label{eq:4}
 \sum_{k=1}^{n-1} x_{k+1} \frac{\partial}{\partial x_k}+f(x_1,\dots, x_n) \frac{\partial}{\partial x_n},
\end{equation}
with $f(x)=O(\|x\|^2)$ where $x=(x_1,\dots,x_n)$. It is said that $X$ has a nilpotent singularity of codimension $n$ at $0$ if the generic condition $\partial^2 f / \partial x_1^2 (0) \not = 0$ is fulfilled. The vector field $X$ itself will be often referred to as a nilpotent singularity of codimension $n$.

Nilpotent singularities of codimension $n$ are generic in families depending on at least $n$ parameters and according with~\cite[Lemma 2.1]{Drubi} we can state the following result:
\begin{lem} \label{lem:Drubi}
Any $n$-parametric generic unfolding of a nilpotent singularity of codimension $n$ in $\mathbb{R}^n$ can be written as
\begin{equation} \label{eq:5}
 \sum_{k=1}^{n-1} x_{k+1} \frac{\partial}{\partial x_k}+\left(\mu_1 + \sum_{k=2}^{n} \mu_k x_k + x^2_1 + h(x,\mu) \right) \frac{\partial}{\partial x_n},
\end{equation}
where $\mu=(\mu_1,\dots,\mu_n)\in\mathbb{R}^n$, $h(0,\mu)=0$, $(\partial h/\partial x_i)(0,\mu)=0$ for $i=1,\ldots,n$, $(\partial^2 h/\partial x_1^2)(0,\mu)=0$, $h(x,\mu)=O(\|(x,\mu)\|^2)$ and $h(x,\mu)=O(\|(x_2,\dots,x_n)\|)$.
\end{lem}
\begin{rem}
Besides the condition $\partial^2 f / \partial x_1^2 (0) \not = 0$ in (\ref{eq:4}), genericity assumptions in Lemma \ref{lem:Drubi} include a transversality condition involving derivatives of the family with respect to parameters.
\end{rem}
The classical techniques of reduction to normal forms could be used to remove terms in the Taylor expansion of $h$ but we do not need to work with simpler expressions. To obtain the results provided in the next sections we will have to impose
\begin{equation}
\label{generic-condition}
\kappa=\frac{\partial^2 h}{\partial x_1 \partial x_2}(0,0)\neq 0,
\end{equation}
as an additional generic assumption.

\subsection{Rescalings and limit families}
\label{rescalingandlimitfamilies}
Generalizing the techniques used in \cite{Dumortier} for dimension three, we rescale varia\-bles and parameters by means of
\begin{eqnarray}
\mu_1 &=& \varepsilon^{2n}\nu_1, \nonumber\\
\mu_k &=& \varepsilon^{n-k+1}\nu_k \quad  \text{for} \ k=2,\dots, n, \label{generalrescaling}\\
 x_k  &=& \varepsilon^{n+k-1}y_k \quad  \text{for} \ k=1,\dots, n, \nonumber
\end{eqnarray}
with $\varepsilon>0$ and $\nu_1^2 + \ldots +\nu^2_n=1$, and also multiply the whole family  by a factor $1/\varepsilon$. In new coordinates and parameters~\eqref{eq:5} can be written as
\begin{equation} \label{eq:familia-nilp}
\sum_{k=1}^{n-1} y_{k+1}\frac{\partial}{\partial y_k}+\big(\nu_1 + \sum_{k=2}^{n} \nu_k y_k + y^2_1 + \varepsilon \kappa y_1 y_2+ O(\varepsilon^2)\big) \frac{\partial}{\partial y_n},
\end{equation}
with $\kappa$ as introduced in (\ref{generic-condition}) and where $y=(y_1, \dots y_n)$ belongs to an arbitrarily big compact in $\mathbb R^n$.

The first step to understand the dynamics arising in generic unfoldings of $n$-dimensional nilpotent singularities of codimension $n$ is the study of the bifurcation diagram of the \emph{limit family}
\begin{equation} \label{eq:familia-limit}
 \sum_{k=1}^{n-1} y_{k+1} \frac{\partial}{\partial y_k}+\big(\nu_1 + \sum_{k=2}^{n} \nu_k y_k + y^2_1 \big) \frac{\partial}{\partial y_n},
\end{equation}
obtained by taking $\varepsilon=0$ in~\eqref{eq:familia-nilp}. Structurally stable behaviours and generic bifurcations in \eqref{eq:familia-limit} should persist in \eqref{eq:familia-nilp} for $\varepsilon>0$ small enough.

If $\nu_1 >0$  then \eqref{eq:familia-limit} has no equilibrium points. Moreover the function
\begin{equation*}
   L(y_1,\dots,y_n)=y_n-\nu_2y_1-\nu_3 y_2 - \ldots - \nu_n y_{n-1}
\end{equation*}
is strictly increasing along the orbits and therefore the maximal compact inva\-riant set is empty. Hence we only need to pay attention to the case $\nu_1 \leq 0$.

On the other hand, up to a change of sign, family~\eqref{eq:familia-limit} is invariant under the transformation
\begin{equation} \label{map}
    \begin{array}{lllll}
    \lefteqn{(\nu, y) \mapsto
    \left(\nu_1,(-1)^{n-1}\nu_2,(-1)^{n-2}\nu_3,\ldots,\nu_{n-1},-\nu_n,\right.}
    \\
    \\
    &&&&\left.(-1)^n y_1,(-1)^{n-1}y_2,(-1)^{n-2}y_3,\ldots,y_{n-1},-y_n\right),
    \end{array}
\end{equation}
with $\nu=(\nu_1,\ldots,\nu_n)$. As a first consequence, the study of bifurcations can be reduced to the region
\begin{equation*}
 \mathcal{R}=\{(\nu_1,\dots,\nu_n) \in \mathbb{S}^{n-1}: \ \nu_1 \leq 0, \ \nu_n\leq 0\}.
\end{equation*}

Moreover, since the limit family is invariant under \eqref{map} up to a change of sign, for parameter values on the set
\begin{equation*}
\mathcal{T}=\{ (\nu_1,\dots,\nu_n)\in \mathbb{S}^{n-1}: \nu_{n-2i}=0 \ \text{with} \ i=0,\ldots,\lfloor (n-2)/2\rfloor  \},
\end{equation*}
where $\lfloor \cdot \rfloor$ denotes the floor function, the correspondent vector fields in the limit family~\eqref{eq:familia-limit} are
time-reversible with respect to the involution
\begin{equation*}
R: (y_1,y_2,y_3,\dots,y_n) \mapsto ( (-1)^n y_1,(-1)^{n-1}y_2,\dots,y_{n-1},-y_n).
\end{equation*}
We said that the manifold $\mathcal{T}$ of  dimension $\lfloor n/2 \rfloor -1$ is the \emph{reversibility set} of the $n$-dimensional nilpotent limit family.

Note that the divergence of the limit family~\eqref{eq:familia-limit} takes the constant value $\nu_n$. Therefore the condition $\nu_n=0$ characterizes a subfamily of volume-preserving vector fields. Assuming that $n$ is even and defining $m=n/2$, for parameter values on the reversibility set the limit family~\eqref{eq:familia-limit} can be written as
\begin{equation} \label{eq:4_anexoII}
 \sum_{k=1}^{n-1} y_{k+1} \frac{\partial}{\partial y_k}+\big(\nu_1 + \sum_{k=1}^{m-1} \nu_{2k+1} y_{2k+1} + y^2_1 \big) \frac{\partial}{\partial y_n}.
\end{equation}
In Appendix~\ref{apendixA} we will prove the following result.

\begin{thm} \label{thm:hamiltonian}
Introducing the new variables $q=S \cdot (y_1,y_3,\ldots,y_{n-1})^t$ and $p=(y_2,y_4,\ldots,y_n)^t$, with
\begin{equation*}
S=\begin{pmatrix}
 -\nu_3     & -\nu_5 & \dots & -\nu_{n-1} & 1 \\
 -\nu_5     &        &\iddots &  \iddots & 0   \\
  \vdots    & \iddots &\iddots & \iddots   & \vdots  \\
-\nu_{n-1}  & \iddots &\iddots &            & \vdots  \\
   1        &   0    & \dots &     \dots  & 0
\end{pmatrix},
\end{equation*}
the family~(\ref{eq:4_anexoII}) transforms into
\begin{equation*}
\frac{\partial H}{\partial p} \, \frac{\partial}{\partial q}-\frac{\partial H}{\partial q} \, \frac{\partial}{\partial p},  \\
\end{equation*}
where
$$H(q,p)=\frac{1}{2}<Sp,p>+V(q).$$
The potential $V$ is defined as
\begin{align*}
V(q)&=
     -\frac{1}{3} q_m^3 - \frac{1}{2} \sum_{k=1}^{m-1}  \nu_{2k+1} b_{k+1} q_m^2 -  \frac{1}{2} \sum_{j=1}^{\lfloor m/2 \rfloor} b_{m-2j+1} q_{m-j}^2 \\
    &- \sum_{k=1}^{m-1} \sum_{i=m-k}^{m-1} \nu_{2k+1} b_{i-m+k+1} q_i q_m - \sum_{j=1}^{\lfloor m/2 \rfloor} \sum_{i=j}^{m-j-1} b_{i} q_i q_{m-j} - \nu_1 q_m,
\end{align*}
where, given $b_1=1$,
\begin{equation*}
b_i=\sum_{\ell=1}^{i-1}\nu_{2(m-i+\ell)+1}b_\ell \quad  \text{for
$ i=2,\ldots,m$.}
\end{equation*}
\end{thm}

\vspace{-0.6cm}
\section{Nilpotent singularity of codimension 4 in $\mathbb{R}^{4}$.}\label{section:nilpotent-4}
We will prove that in any generic unfolding of a nilpotent singularity of codimension four in $\mathbb{R}^4$ there exists a bifurcation hypersurface of homoclinic connections to bifocus equilibria.

Along this section we will take $n=4$ in all the general
expressions introduced in \S\ref{section:nilpotent-n}. It follows
from Lemma \ref{lem:Drubi} that any generic unfolding of the
nilpotent singularity of codimension four in $\mathbb{R}^4$ can be
written as in (\ref{eq:5}). After applying the rescaling
(\ref{generalrescaling}) we get
\begin{equation}
\label{eq:familia-nilp-reescalada-4}
y_{2}\frac{\partial }{\partial y_{1}}+y_{3}\frac{\partial }{\partial y_{2}}%
+y_{4}\frac{\partial }{\partial y_{3}}+\big(\nu _{1}+\nu _{2}y_{2}+\nu
_{3}y_{3}+\nu _{4}y_{4}+y_{1}^{2}+\varepsilon \kappa y_{1}y_{2}+O(\varepsilon ^{2})%
\big)\frac{\partial }{\partial y_{4}},
\end{equation}
with $\nu=(\nu_1,\nu_2,\nu_3,\nu_4)\in \mathbb{S}^3$ and $\varepsilon>0$.

As mentioned in \S\ref{rescalingandlimitfamilies} the first step to understand the dynamics arising in
(\ref{eq:familia-nilp-reescalada-4}) is the study of the limit family
\begin{equation}
\label{eq:familia-limit-4}
y_{2}\frac{\partial }{\partial y_{1}}+y_{3}\frac{\partial }{\partial y_{2}}%
+y_{4}\frac{\partial }{\partial y_{3}}+\big(\nu _{1}+\nu _{2}y_{2}+\nu
_{3}y_{3}+\nu _{4}y_{4}+y_{1}^{2}\big)\frac{\partial }{\partial y_{4}},
\end{equation}
obtained from (\ref{eq:familia-nilp-reescalada-4}) taking
$\varepsilon =0$. As argued in \S\ref{rescalingandlimitfamilies}
one only need to pay attention to parameters in the region $
\mathcal{R}=\{(\nu _{1},\nu _{2},\nu _{3},\nu _{4})\in
\mathbb{S}^{3}:\nu _{1}\leq 0,\,\nu _{4}\leq 0\}$. When
$\nu\in\mathcal{R}$, vector fields in the limit family
(\ref{eq:familia-limit-4}) have equilibrium points $p_{\pm }=(\pm
\sqrt{-\nu_{1}},0,0,0)$ with characteristic equations
\begin{equation}\label{eq:characteristic-4}
r^{4}-\nu _{4}r^{3}-\nu _{3}r^{2}-\nu _{2}r \mp 2 \sqrt{-\nu _{1}}=0.
\end{equation}
Local bifurcations arising in the family were discussed in \cite{DrubiThesi}.
\begin{figure}
\begin{center}
\scalebox{0.7}{
\begin{picture}(0,0)%
\ifpdf\includegraphics{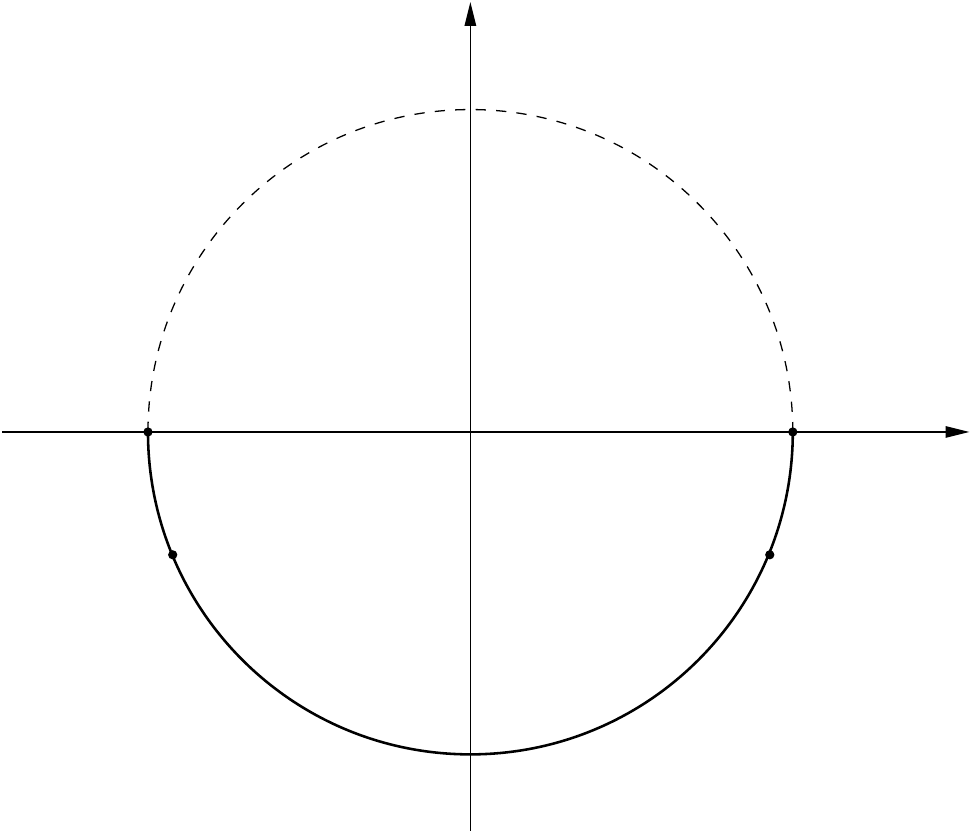}\else\includegraphics{curva_revers.pstex}\fi%
\end{picture}%
\setlength{\unitlength}{3232sp}%
\begin{picture}(5732,4884)(3769,-6103)%
\put(8506,-3661){\makebox(0,0)[lb]{%
\smash{{\SetFigFontNFSS{9}{10.8}{\familydefault}{\mddefault}{\updefault}{\color[rgb]{0,0,0}$\mathrm{BT}$}}}}}%
\put(8371,-4651){\makebox(0,0)[lb]{%
\smash{{\SetFigFontNFSS{9}{10.8}{\familydefault}{\mddefault}{\updefault}{\color[rgb]{0,0,0}$\mathrm{BD}$}}}}}%
\put(4816,-4156){\makebox(0,0)[lb]{%
\smash{{\SetFigFontNFSS{9}{10.8}{\familydefault}{\mddefault}{\updefault}{\color[rgb]{0,0,0}$\mathcal{HH}$}}}}}%
\put(4141,-3661){\makebox(0,0)[lb]{%
\smash{{\SetFigFontNFSS{9}{10.8}{\familydefault}{\mddefault}{\updefault}{\color[rgb]{0,0,0}$\mathrm{HDZ}$}}}}}%
\put(7956,-4156){\makebox(0,0)[lb]{%
\smash{{\SetFigFontNFSS{9}{10.8}{\familydefault}{\mddefault}{\updefault}{\color[rgb]{0,0,0}$\mathcal{SR}$}}}}}%
\put(4411,-4651){\makebox(0,0)[lb]{%
\smash{{\SetFigFontNFSS{9}{10.8}{\familydefault}{\mddefault}{\updefault}{\color[rgb]{0,0,0}$\mathrm{HH}$}}}}}%
\put(9271,-4021){\makebox(0,0)[lb]{%
\smash{{\SetFigFontNFSS{9}{10.8}{\familydefault}{\mddefault}{\updefault}{\color[rgb]{0,0,0}$\nu_3$}}}}}%
\put(5896,-5866){\makebox(0,0)[lb]{%
\smash{{\SetFigFontNFSS{9}{10.8}{\familydefault}{\mddefault}{\updefault}{\color[rgb]{0,0,0}$\mathcal{DF}$}}}}}%
\put(6616,-1411){\makebox(0,0)[lb]{%
\smash{{\SetFigFontNFSS{9}{10.8}{\familydefault}{\mddefault}{\updefault}{\color[rgb]{0,0,0}$\nu_1$}}}}}%
\end{picture}}
\caption{The reversibility curve $\mathcal{T}$ is split into several arcs attending to the type of eigenvalues of the linear part of~\eqref{eq:familia-limit-4} at  $p_-$.} \label{fig:L}
\end{center}
\end{figure}

For parameters on the reversibility curve $\mathcal{T}=\left\{
(\nu _{1},\nu _{2},\nu _{3},\nu _{4})\in \mathbb{S}^{3}:\nu
_{2}=\nu _{4}=0\right\}$ with $\nu _{1}\leq 0$, the characteristic
equations reduces to $r^{4}-\nu _{3}r^2\mp 2\sqrt{-\nu _{1}}=0$.
It follows that the linear part at $p_{+}$ always have a pair of
real eigenvalues and a pair of complex eigenvalues with non-zero
real part. Local behaviour at $p_-$ is richer and it is depicted
in Figure~\ref{fig:L}. Note that we only have to pay attention to
$\nu_1^2+\nu_3^2=1$ with $\nu_1 \leq 0$. It easily follows that
the linear part at $p_-$ has
\begin{itemize}
\item a double zero eigenvalue and eigenvalues $\pm 1$ at $\mathrm{BT}=(0,0,1,0)$,
\item a double zero eigenvalue and a pair of pure imaginary eigenvalues at $\mathrm{HDZ}=(0,0,-1,0)$,
\item two double real eigenvalues $\pm (\nu _{3}/2)^{1/2}$ at $\mathrm{BD}=(\nu_1,0,\nu_3,0)$ with $\nu _{3}^{2}-8\sqrt{-\nu _{1}}=0$ and $\nu_3>0$,
\item two double pure imaginary eigenvalues $\pm i(-\nu_{3}/2)^{1/2}$ at $\mathrm{HH}=(\nu_1,0,\nu_3,0)$ with $\nu _{3}^{2}-8\sqrt{-\nu _{1}}=0$ and $\nu_3<0$,
\item four non-zero real eigenvalues $\pm \lambda_k$, with $k=1,2$ for parameters along the open arc $\mathcal{SR}$ between $\mathrm{BD}$ and $\mathrm{BT}$,
\item four complex eigenvalues with non-zero real part $\rho\pm \omega i$ and $-\rho\pm\omega i$ for parameters along the open arc $\mathcal{DF}$ between $\mathrm{BD}$ and $\mathrm{HH}$,
\item four pure imaginary eigenvalues $\pm\omega_k i$, with $k=1,2$, for parameters along the open arc $\mathcal{HH}$ between $\mathrm{HH}$ and $\mathrm{HDZ}$.
\end{itemize}

From the analysis of the linear part at the equilibrium points it follows that a bifocus is only possible at $p_-$. In order to show the existence of bifocal homoclinic bifurcations in the unfolding of the nilpotent singularity of codimension four in $\mathbb{R}^{4}$ we must study the existence of homoclinic orbits to $p_-$ for parameter values along $\mathcal{T}$.

To study the family (\ref{eq:familia-nilp-reescalada-4}) close to the reversibility curve $\mathcal{T}$ with $\nu_1<0$ it is more convenient to use a directional version of the rescaling (\ref{generalrescaling}) taking $\nu_1=-1$ and $(\nu_2,\nu_3,\nu_4)=(\bar\nu_2,\bar\nu_3,\bar\nu_4)\in\mathbb{R}^3$ to get
\begin{equation}
\label{eq:familia-nilp-reescalada-directional}
y_{2}\frac{\partial }{\partial y_{1}}+y_{3}\frac{\partial }{\partial y_{2}}%
+y_{4}\frac{\partial }{\partial y_{3}}+\big(-1+\bar\nu _{2}y_{2}+\bar\nu
_{3}y_{3}+\bar\nu _{4}y_{4}+y_{1}^{2}+\varepsilon \kappa y_{1}y_{2}+O(\varepsilon ^{2})%
\big)\frac{\partial }{\partial y_{4}}.
\end{equation}
The equilibrium points when $\varepsilon=0$ are given by $q_{\pm}=(\pm 1,0,0,0)$. Note that in fact $q_{\pm}$ are the only equilibrium points even for $\varepsilon>0$ because in (\ref{eq:5}) $h(x,\mu)=O(\|(x_2,\ldots,x_n)\|)$ and this property is preserved by the rescaling. In order to compare with equations already considered in the literature we translate $q_-$ to the origin applying the change of coordinates
$$
x_1=(y_1+1)/2,\qquad x_2=y_2/2^{5/4},\qquad x_3=y_3/2^{6/4},\qquad x_4=y_4/2^{7/4},
$$
to (\ref{eq:familia-nilp-reescalada-directional}) and multiplying by the factor $2^{1/4}$ to obtain
\begin{equation}
\label{eq:familia-nilp-reescalada-directional-translated}
x_{2}\frac{\partial }{\partial x_{1}}+x_{3}\frac{\partial
}{\partial x_{2}} +x_{4}\frac{\partial }{\partial x_{3}}
+\big(-x_{1}+\eta _{2}x_{2}+\eta _{3}x_{3}+\eta
_{4}x_{4}+x_{1}^{2}+\overline{\varepsilon }\kappa
x_{1}x_{2}+O(\bar\varepsilon
^{2})\big)\displaystyle{\frac{\partial }{\partial x_{4}}}
\end{equation}
with $\eta _{2}=2^{-3/4}(\bar\nu_2-\varepsilon\kappa)$, $\eta _{3}=2^{-1/2}\bar\nu_3$, $\eta _{4}=2^{-1/4}\bar\nu_4$ and $\bar\varepsilon
=2^{1/4}\varepsilon$. The equilibrium point $q_{-}$ in ($\ref{eq:familia-nilp-reescalada-directional}$) corresponds to the equilibrium point of ($\ref{eq:familia-nilp-reescalada-directional-translated}$) at the origin. The limit subfamily for $\eta_2=\eta_4=\bar\varepsilon=0$ is now given as
\begin{equation}
x_{2}\frac{\partial }{\partial x_{1}}+x_{3}\frac{\partial }{\partial x_{2}}%
+x_{4}\frac{\partial }{\partial x_{3}}+\big(-x_{1}+\eta _{3}x_{3}+x_{1}^{2}%
\big)\frac{\partial }{\partial x_{4}}.  \label{limitfamilyaftertranslation}
\end{equation}

Writing $u=x_{1}$, ($\ref{limitfamilyaftertranslation}$) is equivalent to the fourth order differential equation
\begin{equation}
u^{(iv)}(t)+Pu^{\prime \prime }(t)+u(t)-u(t)^{2}=0,  \label{4orderequation}
\end{equation}
with $P=-\eta _{3}$. As already mentioned in the introduction, the above equation has been extensively studied in the literature.

In \cite{AmickToland92} authors prove that (\ref{4orderequation})
can be written as a hamiltonian system (as we have stated in
Theorem \ref{thm:hamiltonian} for a more general case) satisfying
the hypothesis required in \cite[Theorem~2]{HoferToland84} to
conclude that, for each $P\leq -2$, there exists an even solution
$u$ with $u(t)\to 0$ when $t\to\pm\infty$ satisfying that $u>0$,
$u'<0$ and $(P/2)u'+u'''>0$ on $(0,\infty)$. They also prove that
for all $P\leq -2$ any such even solution is unique. From
\cite{BuffoniChampneysToland96} it follows that this unique
homoclinic orbit is transversal for the restriction to the level
surface of the hamiltonian function which contains it and,
consequently, it is non degenerate in the sense of Definition
\ref{def:nondegeneratehomoclinicorbit}. Moreover, again in
\cite{AmickToland92}, the persistence of such homoclinic solutions
is argued for $P>-2$ but close enough to $-2$. Variational methods
used in \cite{Buffoni96} allow to prove that at least one
homoclinic solution exists for $P<2$. On the other hand, in
\cite[Section 2]{BuffoniChampneysToland96} authors check all
hypothesis required in \cite[Theorem 4.4]{ChampneysToland93} to
conclude that a Belyakov-Devaney bifurcation takes place at
$P=-2$. It consists in the emerging from the primary homoclinic
solution and for each $n\in\mathbb{N}$ of a finite number of
$n$-modal secondary homoclinics (or $n$-pulses) which cut $n$
times a section transversal to the primary homoclinic orbit
\cite{Devaney76,Belyakov84,BelyakovShilnikov90}.  Heuristic
arguments in \cite{BuffoniChampneysToland96}, supported by
numerical results, show that the non-degenerate $n$-modal
homoclinic orbits arising at $P=-2$ become in degenerate orbits
and disappear gradually when $P$ varies from $P=-2$ to $P=2$
through a cascade of coalescences and bifurcations. In particular,
it is known from \cite{IoossPeroune93} that for $P$ close to $P=2$
there exist at least two even homoclinic solutions and from the
numerical results it seems that no other homoclinic orbits reaches
$P=2$.

All the above results about the existence of homoclinic solutions of (\ref{4orderequation}) can be directly translated to family (\ref{limitfamilyaftertranslation}) and also to the reversible subfamily of (\ref{eq:familia-limit-4}) obtained restricting to parameter values along the previously defined reversibility curve $\mathcal{T}$. For the later case we can conclude that (see Figure \ref{fig:L})
\begin{itemize}
\item for parameter values along $\mathcal{DF}\cup\{\mathrm{BD}\}\cup\mathcal{SR}$ there exists a symmetric homoclinic orbit at $p_-$ which is unique and non degenerate along $\{\mathrm{BD}\}\cup\mathcal{SR}$,
\item $\mathrm{BD}$ is a Belyakov-Devaney bifurcation point,
\item numerical continuation shows that the non degenerate $n$-modal homoclinic orbits arising at $\mathrm{BD}$ become in degenerate orbits and disappear gradually when parameters move along $\mathcal{DF}$ in the direction of $\mathrm{HH}$. Close to that point only two symmetric homoclinic orbits persist.
\end{itemize}

To study the persistence of homoclinic orbits we will consider (\ref{eq:familia-nilp-reescalada-4}) as an unfolding of the Belyakov-Devaney bifurcation point $\mathrm{BD}$. As already mentioned it is better to work with expression (\ref{eq:familia-nilp-reescalada-directional-translated}) for the rescaled unfolding. With respect to parameters $(\eta_2,\eta_3,\eta_4,\bar\varepsilon)$ the point $\mathrm{BD}$ corresponds to $(0,2,0,0)$.  Note that (\ref{eq:familia-nilp-reescalada-directional-translated}) can be written as
\begin{equation}
\label{eq:rescaledfamilyasperturbation}
x^{\prime }=f(x)+g(\lambda ,x),
\end{equation}
where $\lambda =(\lambda _{1},\lambda _{2},\lambda_{3},\lambda _{4})=(\eta _{2},\eta _{3}-2,\eta _{4},\overline{\varepsilon })$,
\begin{equation*}
f(x)=(x_{2},x_{3},x_{4},-x_{1}+2x_{3}+x_{1}^{2})
\end{equation*}
and
\begin{equation*}
g(\lambda,x)=(0,0,0,\lambda_1 x_2+\lambda_2 x_3+\lambda_3 x_4+\lambda_4 \kappa x_1 x_2+O(\lambda_4^2)).
\end{equation*}
As already mentioned, $q_\pm$ are the only equilibrium points of (\ref{eq:familia-nilp-reescalada-directional}) for all $\varepsilon\geq 0$ and hence $g(\lambda,0)=0$ for all $\lambda$. Observe that only bifurcations occurring inside the region of parameters with $\lambda_4>0$ will be observed in the unfolding of the singularity. Family (\ref{eq:rescaledfamilyasperturbation}) fulfills all the hypothesis imposed to (\ref{*1}). In particular, $x^{\prime }=f(x)$
satisfies the following:
\begin{description}
\item[{\rm (BD1)}] It has a first integral
    \begin{equation*}
    H(x_{1},x_{2},x_{3},x_{4})=\frac{1}{2}x_{1}^{2}-\frac{1}{3}%
    x_{1}^{3}-x_{2}^{2}+x_{2}x_{4}-\frac{1}{2}x_{3}^{2}
    \end{equation*}
\item[{\rm (BD2)}] It is time reversible with respect to
\begin{equation*}
R:(x_{1},x_{2},x_{3},x_{4})\mapsto(x_{1},-x_{2},x_{3},-x_{4}).
\end{equation*}
\item[{\rm (BD3)}] The origin is a hyperbolic equilibrium point at which the linear part has a pair of double real eigenvalues $\pm 1$.
\item[{\rm (BD4)}] According to \cite{AmickToland92}, there exists a non degenerate homoclinic orbit $\gamma=\{p(t)=(p_{1}(t),p_{2}(t),p_{3}(t),p_{4}(t))\,:\,t\in\mathbb{R}\}$ to the origin such that $p_{1}(t)$ and $p_{3}(t)$ are even functions and $p_{2}(t)$ and $p_{4}(t)$ are odd functions and, moreover, $p_1>0$, $p_2<0$ and $p_4-p_2>0$ on $(0,\infty)$.
\item[{\rm (BD5)}] According to Proposition \ref{not:3}, since $\gamma$ is non degenerate, both the variational equation $z^{\prime }=Df(p(t))z$ and its adjoint $z^{\prime}=-Df(p(t))^{\ast }z$ has a unique non trivial linearly independent bounded solution. The function $\varphi (t)=f(p(t))$ is a bounded solution of the variational equation and
    \begin{equation*}
    \psi (t)=\nabla H(p(t))=(p_{1}(t)-p_{1}(t)^{2},p_{4}(t)-2p_{2}(t),-p_{3}(t),p_{2}(t))
    \end{equation*}%
    is a bounded solution of adjoint equation.
\end{description}

Finally, let us  consider the bifurcation equation for homoclinic
solutions $\xi ^{\infty }(\lambda )=0$, with
$\xi^{\infty}:\Lambda\to\mathbb{R}$ and
$\Lambda\subset\mathbb{R}^4$ a neighbourhood of the origin, as
introduced in Lemma \ref{lem:eq-bif}. It follows from
Theorem~\ref{thm:bif} that under the generic condition
$$\nabla
\xi ^{\infty }(0)=(\xi _{\lambda_1},\xi
_{\lambda_2},\xi_{\lambda_3},\xi _{\lambda_4})\not=0$$ where
$$
\xi_{\lambda_i}=\int_{-\infty }^{\infty }\langle
\psi(t),\frac{\partial g}{\partial \lambda_i}(0,p(t))\rangle\,dt,
$$
then (\ref{eq:rescaledfamilyasperturbation}) has homoclinic orbits
(continuation of $\gamma$) for parameters on a hypersurface
$\mathcal{H}om$ with tangent subspace at $\lambda=0$ given by
\begin{equation} \xi _{\lambda_1}\lambda_1+\xi
_{\lambda_2}\lambda_2+\xi _{\lambda_3}\lambda_3+\xi
_{\lambda_4}\lambda_4 =0.  \label{eq:bi2}
\end{equation}
Note that
\begin{align*}
    \xi_{\lambda_1}&=\int_{-\infty }^{\infty }p_{2}^2(t)\,dt, &
    \xi_{\lambda_2}&=\int_{-\infty }^{\infty }p_2(t)p_{3}(t)\,dt, \\
    \xi_{\lambda_3}&=\int_{-\infty }^{\infty }p_2(t)p_{4}(t)\,dt, &
    \xi_{\lambda_4}&=\int_{-\infty }^{\infty }\kappa p_{1}(t)p_{2}^2(t)\,dt.
\end{align*}
Clearly $\xi _{\lambda_1}\neq 0$. Since $p_{2}p_{3}$ is an odd function $\xi _{\lambda_2} = 0$. Integrating by parts one gets $\xi _{\lambda_3}=-\int_{-\infty }^{\infty }p_{3}(t)^2\,dt \neq 0$. Finally, since
$p_1$ is a positive function, we also get that $\xi_{\lambda_4}\neq 0$. Therefore the tangent subspace (\ref{eq:bi2}) intersects $\lambda_4=0$ transversely. Consequently $\mathcal{H}om$ also meets $\lambda_4=0$ transversely.

Now we have to study the eigenvalues at the equilibrium point in order to determine which types of homoclinic orbits can be unfolded by the singularity. Since for $\lambda=0$ the linear part at $x=0$ has a pair of double real eigenvalues $\pm 1$ and $\mathrm{dim}W^s(0)=\mathrm{dim}W^u(0)=2$, for all $\lambda$ small enough, then we can expect three different types of equilibrium: a focus-focus (bifocus), a node-node or a focus-node. It easily follows that the characteristic polynomial at $x=0$ is given by
$$
Q(r,\lambda)=r^4-D(\lambda)r^3-C(\lambda)r^2-B(\lambda)r-A(\lambda),
$$
with
$$
\begin{array}{ll}
A(\lambda)=-1+O(\lambda_4^2) & B(\lambda)=\lambda_1+O(\lambda_4^2)
\\
C(\lambda)=2+\lambda_2+O(\lambda_4^2) & D(\lambda)=\lambda_3+O(\lambda_4^2).
\end{array}
$$
The condition for an improper node is given by the discriminant equations
$$
Q(r,\lambda)=0,
\qquad
\frac{\partial Q}{\partial r}(r,\lambda)=0.
$$
Note that $(r,\lambda)=(\pm 1,0)$ are both solutions of the discriminant equations. Now it follows from a straightforward application of the Implicit Function Theorem that there exist two hypersurfaces $\mathcal{D^-}$ and $\mathcal{D^+}$ through the origin in the parameter space such that, for parameter values on $\mathcal{D^-}$ (resp. $\mathcal{D^+}$) the equilibrium point at the origin has a double negative (resp. positive) real eigenvalue. Moreover the respective tangent subspaces at $\lambda=0$ are $\lambda_1-\lambda_2+\lambda_3=0$ and $\lambda_1+\lambda_2+\lambda_3=0$. Let $N_{\mathcal{H}om}=(\xi_{\lambda_1},\xi_{\lambda_2},\xi_{\lambda_3},\xi_{\lambda_4})$, $N_{\mathcal{D^-}}=(1,-1,1,0)$ and $N_{\mathcal{D^+}}=(1,1,1,0)$ be the normal vectors to the tangent spaces of $\mathcal{H}om$, $\mathcal{D^-}$ and $\mathcal{D^+}$ at $\lambda=0$, respectively. Moreover denote $N_{\lambda_4=0}=(0,0,0,1)$. Since $\mathrm{rank}(N_{\mathcal{H}om},N_{\mathcal{D^-}},N_{\lambda_4=0})=3$, there exists a surface $\mathcal{H}om^-=\mathcal{H}om\cap\mathcal{D^-}$ transverse to $\lambda_4=0$ of homoclinic orbits to an equilibrium point with a double negative real eigenvalue. Moreover, since $\mathrm{rank}(N_{\mathcal{H}om},N_{\mathcal{D^+}},N_{\lambda_4=0})=3$, there exists a surface $\mathcal{H}om^+=\mathcal{H}om\cap\mathcal{D^+}$ transverse to $\lambda_4=0$ of homoclinic orbits to an equilibrium point with a double positive real eigenvalue. On the other hand $\mathrm{rank}(N_{\mathcal{H}om},N_{\mathcal{D^-}},N_{\mathcal{D^+}},N_{\lambda_4=0})=4$ if and only if $\xi_{\lambda_1}-\xi_{\lambda_3}\neq 0$. But, taking into account that $p_2$ and $p_4$ are odd functions and also that $p_2<0$ and $p_4-p_2>0$ on $(0,\infty)$ it follows that
$$
\xi_{\lambda_1}-\xi_{\lambda_3}=\int_{-\infty}^{\infty}p_2(t)(p_2(t)-p_4(t))\,dt>0.
$$
Hence we can conclude that $\mathrm{rank}(N_{\mathcal{H}om},N_{\mathcal{D^-}},N_{\mathcal{D^+}},N_{\lambda_4=0})=4$. Therefore there exists a curve $\mathcal{H}om^{\pm}=\mathcal{H}om\cap\mathcal{D^-}\cap\mathcal{D^+}$ transverse to $\lambda_4=0$ of homoclinic orbits to an equilibrium point with a pair of double real eigenvalues one positive and the other negative.

Summarizing, we have proved the following result
\begin{thm}
\label{thm:hombif}
In a neighbourhood of $\lambda=0$ there exists a bifurcation hypersurface $\mathcal{H}om$ corresponding to parameter values for which (\ref{eq:rescaledfamilyasperturbation}) has homoclinic orbits to the origin. Moreover there exist two bifurcation surfaces $\mathcal{H}om^-$ and $\mathcal{H}om^+$ contained in $\mathcal{H}om$ corresponding to parameter values for which the origin has a double negative and positive, respectively, real eigenvalue. The surfaces $\mathcal{H}om^+$ and $\mathcal{H}om^-$ intersect transversely along a curve $\mathcal{H}om^{\pm}$ corresponding to parameter values for which the origin has a pair of double real eigenvalues $\{r_1,r_2\}$ with $r_1<0<r_2$. $\mathcal{H}om^-\cup \mathcal{H}om^+$ splits $\mathcal{H}om$ into four regions:
\begin{enumerate}
\renewcommand{\labelenumi}{(\roman{enumi})}
\item $\mathcal{H}om_{FF}$: homoclinic orbits to a focus-focus equilibrium (bifocus case),
\item $\mathcal{H}om_{N^+F^-}$: homoclinic orbits to a (repelling) node-(attracting) focus equilibrium,
\item $\mathcal{H}om_{F^+N^-}$: homoclinic orbits to a (repelling) focus-(attracting) node equilibrium,
\item $\mathcal{H}om_{NN}$: homoclinic orbits to a node-node equilibrium.
\end{enumerate}
All bifurcations are transverse to $\lambda_4=0$.
\end{thm}

Since all bifurcations are transverse to $\lambda_4=0$ they are also present in the unfolding of the nilpotent singularity of codimension four. Particularly, Theorem B follows as a corollary of Theorem \ref{thm:hombif}.

\begin{rem}
Recall that the bifurcation point $\lambda=0$ in (\ref{eq:rescaledfamilyasperturbation}) corresponds to the Belyakov-Devaney bifurcation point $\mathrm{BD}$ in (\ref{eq:familia-nilp-reescalada-4}). In particular, the hypersurface of parameters corresponding to homoclinic orbits to a node-node equilibrium point in family (\ref{eq:familia-nilp-reescalada-4}) cuts $\varepsilon=0$ along the curve $\mathcal{SR}$. This follows by restricting the bifurcation analysis to $\lambda_4=0$.
\end{rem}

\section{Nilpotent singularity of codimension 3 in $\mathbb{R}^{3}$.}\label{section:nilpotent-3}
We will prove that in any generic unfolding of the nilpotent singularity of codimension three in $\mathbb{R}^3$ there exists a one-side bifurcation curve of topological Bykov cycles.

Along this section we will take $n=3$ in all the general expressions introduced in \S\ref{section:nilpotent-n}. It follows from Lemma \ref{lem:Drubi} that any generic unfolding of the nilpotent singularity of codimension three in $\mathbb{R}^3$ can be written as in (\ref{eq:5}). After applying the rescaling (\ref{generalrescaling}) we get
\begin{equation}
\label{eq:familia-nilp-3}
y_{2}\frac{\partial }{\partial y_{1}}+y_{3}\frac{\partial }{\partial y_{2}}+%
\big(\nu _{1}+\nu _{2}y_{2}+\nu _{3}y_{3}+y_{1}^{2}+\varepsilon \kappa
y_{1}y_{2}+O(\varepsilon ^{2})\big)\frac{\partial }{\partial y_{3}}.
\end{equation}
with $\nu=(\nu_1,\nu_2,\nu_3)\in \mathbb{S}^2$ and $\varepsilon>0$.

As mentioned in \S\ref{rescalingandlimitfamilies} the first step to understand the dynamics arising in
(\ref{eq:familia-nilp-3}) is the study of the limit family
\begin{eqnarray*}
y_{2}\frac{\partial }{\partial y_{1}}+y_{3}\frac{\partial }{\partial y_{2}}+%
\big(\nu _{1}+\nu _{2}y_{2}+\nu _{3}y_{3}+y_{1}^{2}\big)\frac{\partial }{%
\partial y_{3}}.
\end{eqnarray*}

The above family has already been treated in the literature and discussions about several aspects of the dynamics can be seen in \cite{Dumortier,dumibakok2001,dumibakok2006,IbanezRodriguez} and refe\-rences there included. As in the case of the nilpotent singularity of codimension four in $\mathbb{R}^4$ we are interested in the dynamics close to the reversibility curve $\mathcal{T}=\{(\nu_1,\nu_2,\nu_3)\in\mathbb{S}^2\,:\,\nu_3=0\}.$
Particularly, we will pay attention to parameters with $\nu_1<0$ and $\nu_2 <0$.

To study family (\ref{eq:familia-nilp-3}) close to the reversibility curve it is more convenient to use a directional version of the rescaling (\ref{generalrescaling}) taking $\nu_2=-1$ and $(\nu_1,\nu_3)=(\bar\nu_1,\bar\nu_3)\in\mathbb{R}^2$ to get
\begin{eqnarray*}
y_{2}\frac{\partial }{\partial y_{1}}+y_{3}\frac{\partial }{\partial y_{2}}+%
\big(\bar\nu _{1}-y_{2}+\bar\nu _{3}y_{3}+y_{1}^{2}+\varepsilon \kappa
y_{1}y_{2}+O(\varepsilon ^{2})\big)\frac{\partial }{\partial y_{3}}.
\end{eqnarray*}
Moreover, in order to use results already present in the literature, we introduce new variables $(x_1,x_2,x_3)=-2(y_1,y_2,y_3)$ and write
$-2\bar\nu_1=c^2$ when $\nu_1<0$ to obtain
\begin{equation}
\label{eq:familia-nilp-reescalada-directional-michelson}
x_{2}\frac{\partial }{\partial x_{1}}+x_{3}\frac{\partial }{\partial x_{2}}+%
\big(c^2-x_{2}+\bar\nu _{3}x_{3}-\frac{1}{2}x_{1}^{2}-2\varepsilon \kappa
x_{1}x_{2}+O(\varepsilon ^{2})\big)\frac{\partial }{\partial x_{3}}.
\end{equation}
In the above expression, taking the limit case $\varepsilon=0$ and also $\bar\nu_3=0$ we get the $1$-parameter family
\begin{eqnarray*}
x_{2}\frac{\partial }{\partial x_{1}}+x_{3}\frac{\partial }{\partial x_{2}}+%
\big(c^2-x_{2}-\frac{1}{2}x_{1}^{2}\big)\frac{\partial }{\partial x_{3}}.
\end{eqnarray*}
As mentioned in the introduction the above family has been extensively studied in the literature. It is commonly referred as the \emph{Michelson system} and it has the following properties:
\begin{description}
\item[{\rm (M1)}] For all $c\geq 0$ the system has only two equilibrium points $Q_\pm=(\pm\sqrt{2}c,0,0)$ where the characteristic equation of the linear part is given by $r^3+r\mp\sqrt{2}c$.
\item[{\rm (M2)}] The eigenvalues at $Q_+$ (resp. $Q_-$) are $\lambda$ and $-\rho\pm i\omega$ (resp. $-\lambda$ and $\rho\pm i\omega$) with $\lambda>0$, $\rho>0$ and $\omega\neq 0$. Therefore $\mathrm{dim}W^s(Q_+)=\mathrm{dim}W^u(Q_-)=2$. Moreover, since the Michelson system has zero divergence, $\lambda-2\rho=0$ and hence $\rho<\lambda$.
\item[{\rm (M3)}] It is time-reversible with respect to the involution
$$
R:(x_1,x_2,x_3)\mapsto(-x_1,x_2,-x_3).
$$
\item[{\rm (M4)}] It follows from \cite{Kuramoto} that the Michelson system has a solution $p(t)$ given by
\begin{equation}
p_{1}(t)=\alpha (-9\tanh \beta t+11\tanh ^{3}\beta
t),~p_{2}(t)=p_{1}^{\prime }(t),~p_{3}(t)=p_{1}^{\prime \prime }(t)
\label{14*}
\end{equation}
with $\alpha =15\sqrt{11/19^{3}}$ and $\beta =\sqrt{11/19}/2$, when $c=c_{k}=\sqrt{2}\alpha$. It
parametrizes an orbit $\Gamma _{1}$ along which two branches of the 1-dimensional
invariant manifolds coincide. The orbit $\Gamma_1$ is invariant by the reversibility, that is,
\begin{equation*}
p(-t)=(p_1(-t),p_2(-t),p_3(-t))=R p(t)=(-p_1(t),p_2(t),-p_3(t)).
\end{equation*}%
Hence $p_1$ and $p_3$ are odd functions and $p_2$ is an even function.
\item[{\rm (M5)}] It follows from \cite{IbanezRodriguez} that when $c=c_{k}$ there also exists an orbit $\Gamma_2=W^s(Q_+) \cap W^u(Q_-)$. Moreover the intersection is topologically transversal.
\end{description}

Putting together (M4) and (M5) it follows that
\begin{description}
\item[{\rm (M6)}] When $c=c_k$ the Michelson system has a topological Bykov cycle.
\end{description}

Let us recall the notion of topological Bykov cycle.

\begin{defi}
\label{def:Tpoint}
Consider a vector field in $\mathbb{R}^{3}$ with two hyperbolic equilibrium points $p_{\pm }$ satisfying $\mathrm{dim} W^{s}(p_{+})=\mathrm{dim} W^{u}(p_{-})=2$. Any heteroclinic cycle consisting of two heteroclinic orbits $\Gamma _{1}\subseteq W^{u}(p_{+})\cap W^{s}(p_{-})$ and $
\Gamma _{2}\subseteq W^{s}(p_{+})\cap W^{u}(p_{-})$ between the equilibrium points $p_{\pm }$ such that the intersection along $\Gamma _{2}$ is transversal is called a T-point. When at $p_{\pm }$ the linear part has complex eigenvalues the T-point is called a Bykov cycle. In both cases, if the intersection along $\Gamma_2$ is only topologically transversal we refer to a topological T-point or a topological Bykov cycle.
\end{defi}

\begin{rem} Let us recall the notion of topological transversality. Let $D$ be a {\rm 2}-dimensional open disk transversal to the flow at some point $T\in \Gamma _{2}$. Let $S$ and $U$ be the connected components of
$W^{u}(P_{-})\cap D$ and $W^{s}(P_{+})\cap D$, respectively, containing the point $T$. If the diameter of $D$ is small enough both
$S$ and $U$ split $D$ into two connected components. The connection $\Gamma _{2}$ is said \emph{topologically transversal} if the two connected components of $S\setminus \{T\}$ belong to different connected components of $D\setminus U$.
\end{rem}

To study the persistence of the Bykov cycle we will consider (\ref{eq:familia-nilp-reescalada-directional-michelson}) as an unfolding of the Michelson system at the Kuramoto point $c=c_k$ introduced in (M4). Note that (\ref{eq:familia-nilp-reescalada-directional-michelson}) can be written as
\begin{equation}
\label{eq:rescaledfamilyasperturbation3}
x'=f(x)+g(\lambda,x),
\end{equation}
where $\lambda=(\lambda_1,\lambda_2,\lambda_3)=(c^2-c_k^2,\bar\nu_3,\varepsilon)$,
$$
f(x)=(x_2,x_3,c_k^2-x_2-\frac{1}{2}x_1^2)
$$
and
$$
g(\lambda,x)=(0,0,\lambda_1+\lambda_2 x_3-2\lambda_3 x_1 x_2+O(\lambda_3^2)).
$$
Bifurcations occurring inside the region of parameters with $\lambda_3>0$ will be observed in the unfolding of the singularity. Family (\ref{eq:rescaledfamilyasperturbation3}) fulfills all the hypothesis imposed to (\ref{*1}). Particularly $x'=f(x)$ satisfies all properties from (M1) to (M6).

The heteroclinic orbit $\Gamma _{1}$ is non degenerate of codimension two in the sense of Definition \ref{def:nondegeneratehomoclinicorbit}. Hence the variational equation $z^{\prime}(t)=Df(p(t))z(t)$ has a unique (up to multiplicative constants) bounded solution $f(p(t))$ whereas the adjoint variational equation $w^{\prime}(t)=-Df(p(t))^{\ast }w(t)$ has a pair of li\-nearly independent bounded solutions, again according to Proposition \ref{not:3}.
Let $\varphi(t)=(\varphi_1(t),\varphi_2(t),\varphi_3(t))$ and $\psi(t)=(\psi_1(t),\psi_2(t),\psi_3(t))$ two of such solutions. Since $\varphi(t)\wedge\psi(t)$ is a bounded solution of the variational equation it follows that the plane determined by $\varphi(t)$ and $\psi(t)$ is orthogonal to $f(p(t))$ for all values of $t$. Therefore, all solutions of the adjoint variational equations with initial conditions on $f(p(0))^\perp$ are bounded solutions.

On the other hand, it easily follows that $w^{\prime}(t)=-Df(p(t))^{\ast }w(t)$ is inva\-riant under the involutions $(w_1,w_2,w_3)\mapsto (-w_1,w_2,-w_3)$ and $(w_1,w_2,w_3)\mapsto (w_1,-w_2,w_3)$ and the time reverse $t\mapsto -t$. Therefore, taking $\varphi(0)=(0,-1,0)$ and $\psi(0)=(1-c_k^2/p_2(0),0,1)$ we can conclude that $\varphi(t)$ and $\psi(t)$ are bounded solutions of the adjoint variational equation and that $\varphi_1$, $\varphi_3$ and $\psi_2$ are odd functions whereas $\varphi_2$, $\psi_1$ and $\psi_3$ are even functions.

\begin{figure}[h]
\begin{center}
\hspace*{-1cm}
\ifpdf\includegraphics[width=0.6\textwidth,keepaspectratio=true]{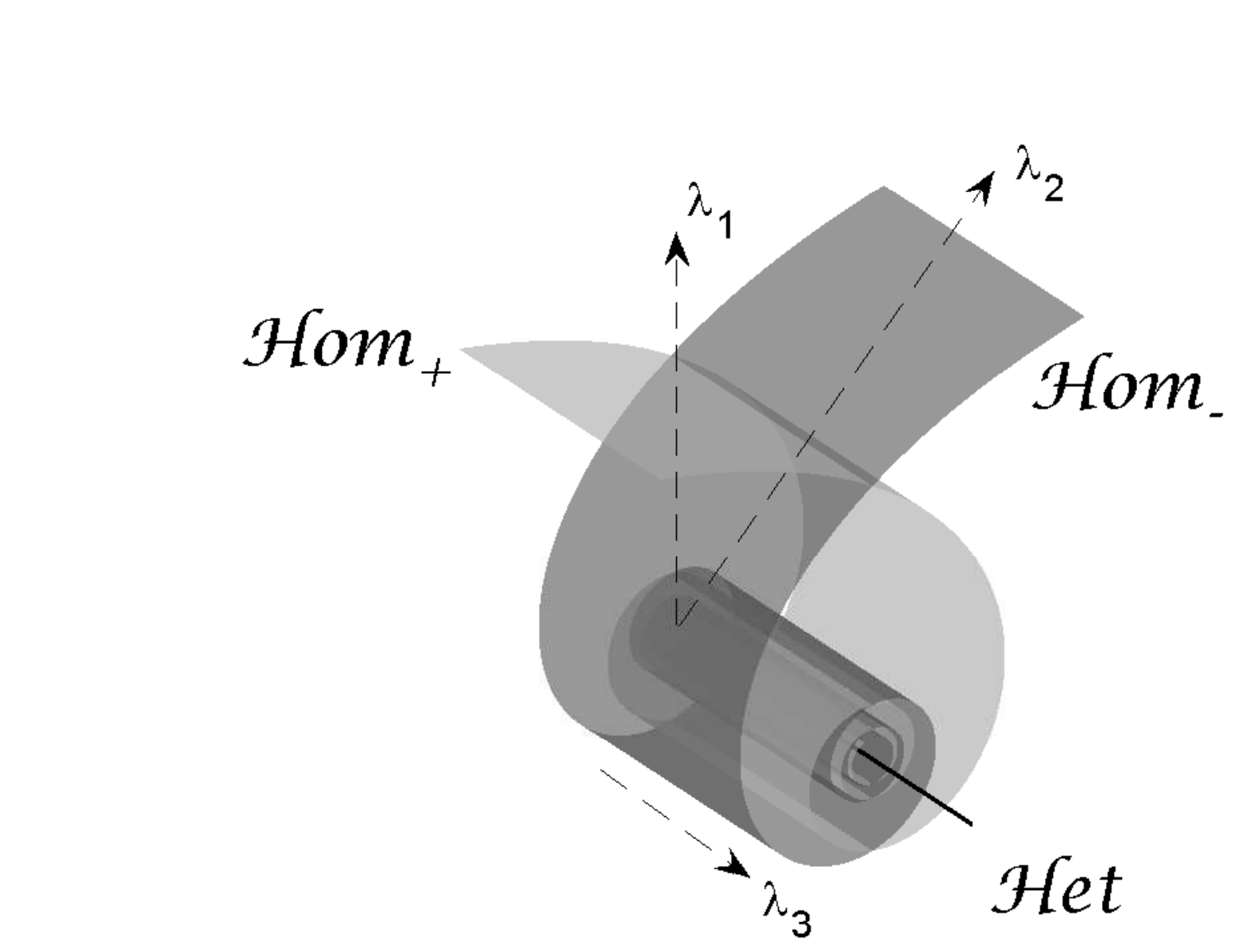}\else%
\includegraphics[width=0.6\textwidth,keepaspectratio=true]{byk.eps}\fi%
\end{center}
%\vspace*{-0.5cm}
\caption{Sketch showing the bifurcation curve $\mathcal{H}et$ to topological Bykov cycles in family (\ref{eq:rescaledfamilyasperturbation3}) and the spirals of bifurcation to homoclinic orbits $\mathcal{H}om_+$ and $\mathcal{H}om_-$ for $Q_+$ and $Q_-$, respectively. Note that when $\lambda_3=0$, the existence of $\mathcal{H}om_-$ follows from the symmetry with respect to (\ref{map}).}
\label{fig:H}
\end{figure}

Consider now the bifurcation equation for homoclinic solutions
$\xi ^{\infty }(\lambda )=0$, with
$\xi^{\infty}=(\xi_1,\xi_2):\Lambda\to\mathbb{R}^2$ and
$\Lambda\subset\mathbb{R}^3$ a neighbourhood of the origin, as
introduced in Lemma \ref{lem:eq-bif}. It follows from
Theorem~\ref{thm:bif} that, under the generic condition
$\mathrm{rank}D_\lambda\xi^\infty(0)=2$, where
$$
D_\lambda\xi^\infty(0)=
\left(
\begin{array}{llll}
\xi_{1,\lambda_1}&\xi_{1,\lambda_2}&\xi_{1,\lambda_3}
\\
\xi_{2,\lambda_1}&\xi_{2,\lambda_2}&\xi_{2,\lambda_3}
\end{array}
\right)
$$
with
\begin{align}
    \xi_{1,\lambda_1}&=\int_{-\infty }^{\infty }\varphi_3(t)\,dt,       & \xi_{2,\lambda_1}&=\int_{-\infty }^{\infty }\psi_3(t)\,dt,\nonumber
    \\
    \xi_{1,\lambda_2}&=\int_{-\infty }^{\infty }\varphi_3(t)p_3(t)\,dt, & \xi_{2,\lambda_2}&=\int_{-\infty }^{\infty }\psi_3(t)p_3(t)\,dt, \label{coefficients}
    \\
    \xi_{1,\lambda_3}&=\int_{-\infty }^{\infty }-2\kappa\varphi_3(t)p_1(t)p_2(t)\,dt,
    &
    \xi_{2,\lambda_3}&=\int_{-\infty }^{\infty }-2\kappa\psi_3(t)p_1(t)p_2(t)\,dt,\nonumber
\end{align}
then (\ref{eq:rescaledfamilyasperturbation3}) has heteroclinic orbits (continuation of $\Gamma_1$) for parameters on a bifurcation curve $\mathcal{H}et$ with tangent subspace at $\lambda=0$ given by the intersection of the planes
$$
\begin{array}{l}
\xi _{1,\lambda_1}\lambda_1+\xi _{1,\lambda_2}\lambda_2+\xi _{1,\lambda_3}\lambda_3=0
\\
\xi _{2,\lambda_1}\lambda_1+\xi _{2,\lambda_2}\lambda_2+\xi _{2,\lambda_3}\lambda_3=0.
\end{array}
$$
From the parities of $p_1$, $p_2$, $p_3$, $\varphi_3$ and $\psi_3$ it follows that
$$
    \xi_{1,\lambda_1}=\xi_{2,\lambda_2}=\xi_{2,\lambda_3}=0.
$$
On the other hand, in Appendix~\ref{apendixB} we will show that
$$
    \xi_{1,\lambda_2} \neq 0, \quad
    \xi_{1,\lambda_3} \neq 0, \quad
    \xi_{2,\lambda_1} \neq 0.
$$
Hence, it is a straightforward application of the Implicit Function Theorem that, indeed, there exists a bifurcation curve $\mathcal{H}et$ of heteroclinic connections along the one dimensional invariant manifolds. Moreover it easily follows that
the tangent space at $\lambda=0$ is generated by a vector $(0,-\xi_{1,\lambda_3}/\xi_{1,\lambda_2},1)$ and hence $\mathcal{H}et$ intersects $\lambda_3=0$ transversely. Since $\Gamma_2$ is a topologically transverse intersection, $\mathcal{H}et$ is a bifurcation curve of topological Bykov cycles. This concludes the proof of Theorem C which was stated in the introduction.

As already mentioned there is a Shil'nikov bifurcation surface $\mathcal{H}om_+$ shaped as a scroll around $\mathcal{H}et$ (see Figure \ref{fig:H}) corresponding to parameter values for which the system has a Shil'nikov homoclinic orbit to $Q_+$. Note that the Shil'nikov condition is open and hence it follows from (M2). Moreover since the trace of the linear part at $Q_+$ is given by $\lambda_2$ the dissipative condition is also satisfied in $\mathcal{H}om_+\cap\{\lambda\in\mathbb{R}^3:\lambda_2<0\}$. Hence strange attractors exist for parameter values on positive Lebesgue measure set.
\begin{rem}
In~\cite{dumibakok2006} the existence of subsidiary Bykov cycles in the Michelson system for values of $c$ close to $c_k$ is also discussed. Moreover, in \cite{Teixeira} the accumulation of Bykov cycles when $c\to 0$ is also argued. In all cases the heteroclinic connections have the symmetry properties that we have just used. It should be possible to extend our result to conclude that in (\ref{eq:familia-nilp-reescalada-directional-michelson}), and consequently in (\ref{eq:familia-nilp-3}), there exist more bifurcation curves to Bykov cycles and particularly an infinite sequence of such type of bifurcation curves. Nevertheless, the generic conditions on the bifurcation equation need to be checked.
\end{rem}

\section*{Acknowledgements}
The authors have been partially supported by the project
MTM2008-06065. The first author has also been supported by the FPU
grant AP2007-031035. We are grateful to Prof. Pablo P\'erez for
his assistance with numerical computations.
\newpage\appendix
\renewcommand{\thesection}{A}
\renewcommand{\theequation}{A.\arabic{equation}}
\setcounter{equation}{0} \setcounter{thm}{0}
\section{Proof of Theorem~\ref{thm:hamiltonian}}
\label{apendixA} We will use the following technical result:
\begin{lem} \label{anex:lema1}
Given a symmetric upper anti-triangular matrix
\begin{equation*}
 A=
\begin{pmatrix}
 a_m     & a_{m-1}  & \ldots    & a_2     & 1 \\
 a_{m-1} &          & \iddots   & \iddots  & 0 \\
 \vdots  & \iddots  & \iddots   &\iddots&\vdots\\
   a_2   &  \iddots &  \iddots  &      & \vdots\\
    1    &    0     & \ldots & \ldots  & 0
\end{pmatrix},
\end{equation*}
$A^{-1}$ is a lower anti-triangular symmetric matrix
\begin{equation*}
 A^{-1}=
\begin{pmatrix}
 0     & \ldots   & \ldots  & 0  & 1   \\
\vdots &          & \iddots & \iddots & b_2 \\
\vdots &  \iddots & \iddots & \iddots & \vdots\\
 0     &  \iddots & \iddots &         & b_{m-1}\\
 1     &  b_2     & \ldots  & b_{m-1} & b_{m}
\end{pmatrix}
\end{equation*}
where, given $b_1=1$,
\begin{equation*}
b_i=-\sum_{\ell=1}^{i-1}a_{i-\ell+1}b_\ell
\end{equation*}
for $i=2,\ldots,m$.
\end{lem}
\begin{proof}
Let $P$ be an anti-diagonal matrix with all entries equal to $1$.
Hence
\begin{equation*}
 L=PA=
\begin{pmatrix}
 1      & 0     & \ldots   &  \ldots     & 0 \\
 a_2    & \ddots  & \ddots   &          & \vdots \\
 \vdots & \ddots  & \ddots   &\ddots  & \vdots \\
 a_{m-1} &         & \ddots   &  \ddots   & 0  \\
 a_m   &  a_{m-1}  & \ldots   &  a_2    & 1
\end{pmatrix}.
\end{equation*}
is a lower triangular matrix.  Therefore, $L^{-1}=(b_{i,j})$ is  also a lower triangular matrix and hence $A^{-1}=L^{-1}P^{-1}=L^{-1}P$ is a lower anti-triangular matrix. In fact, using
the well know formulas for the calculation of the inverse of a triangular matrix, it follows that, for all $j=1,\ldots,m$
\begin{align*}
b_{j,j}&=1,\\
b_{i,j}&=0 \quad \text{for all} \ i=1,\ldots,j-1,\\
b_{i,j}&=-\sum_{\ell=j}^{i-1} a_{i-\ell+1} b_{\ell,j} \quad \text{for all} \ i=j+1,\ldots,m.
\end{align*}

On the other hand, $b_{i,j}=b_{i+1,j+1}$ for all $i=j+1,\ldots,m-1$. Indeed it is clear for $i=j+1$. For $i=j+2, \ldots, m-1$ we can argue by induction.

Finally, by defining $b_i=b_{i,1}$  for all $i=1,\ldots,m$, and calculating $A^{-1}=L^{-1}P$ the proof is finished.
\end{proof}

It follows from Lemma~\ref{anex:lema1} that
\begin{equation*}
 S^{-1}=
\begin{pmatrix}
 0     & \ldots   & \ldots  & 0  & 1   \\
\vdots &          & \iddots & \iddots & b_2 \\
\vdots &  \iddots & \iddots & \iddots & \vdots\\
 0     &  \iddots & \iddots &         & b_{m-1}\\
 1     &  b_2     & \ldots  & b_{m-1} & b_{m}
\end{pmatrix},
\end{equation*}
where, defining $b_1=1$,
\begin{equation*}
b_i=\sum_{\ell=1}^{i-1}\nu_{2(m-i+\ell)+1}b_\ell \quad \text{for} \ i=2,\ldots,m.
\end{equation*}
or equivalently
\begin{equation} \label{eq:recurrencia}
b_{m-j+1}=\sum_{\ell=1}^{m-j}\nu_{2(j+\ell-1)+1}b_\ell=\sum_{k=j}^{m-1} \nu_{2k+1}b_{k-j+1} \quad \text{for} \ j=1,\ldots,m-1.
\end{equation}
Writing family~\eqref{eq:4_anexoII} in the new variables we get
\begin{equation*}
 Sp \, \frac{\partial}{\partial q } + \sum_{k=1}^{m-1} \big(\sum_{i=m-k}^m b_{i-m+k+1} q_i \big)\, \frac{\partial}{\partial p_k } + \big(\nu_1+  \sum_{k=1}^{m-1} \nu_{2k+1} \dot p_k + q_m^2\big)\, \frac{\partial}{\partial p_m }.
\end{equation*}
To obtain a function $V(q)$ such that $\dot p= - \nabla V(q)$ we need $-\partial V/\partial q_i = \dot p_i$ for all $i=1,\ldots,m$. In particular
\begin{equation*}
-\frac{\partial V}{\partial q_m}= \nu_1+\sum_{k=1}^{m-1} \nu_{2k+1} \dot p _k+q_m^2,
\end{equation*}
and therefore
\begin{align*}
-V(q) &= \nu_1 q_m +
 \sum_{k=1}^{m-1}\nu_{2k+1} \big( \frac{1}{2} b_{k+1} q_m^2+ \sum_{i=m-k}^{m-1} b_{i-m+k+1} q_i q_m  \big) \\
 &+\frac{1}{3} q_m^3 +\varphi_{m-1}(q_1,\ldots,q_{m-1}).
\end{align*}
From the identity $-\partial V / \partial q_{m-1} = \dot p_{m-1}$ and taking into account the equation~\eqref{eq:recurrencia} we get
\begin{equation*}
 \frac{\partial\varphi_{m-1}}{\partial q_{m-1}} =
 \sum_{i=1}^m b_{i} q_i - \sum_{k=1}^{m-1} \nu_{2k+1} b_{k} q_m = \sum_{i=1}^{m-1} b_{i} q_i
\end{equation*}
and therefore
$$
\varphi_{m-1}(q_1,\ldots,q_{m-1})=\frac{1}{2}b_{m-1}q_{m-1}^2 + \sum_{i=1}^{m-2} b_i q_i q_{m-1}+ \varphi_{m-2}(q_1,\ldots,q_{m-2}).
$$
Since $-\partial V / \partial q_{m-2}=\dot p_{m-2}$, a similar computation leads to
\begin{equation*}
 \frac{\partial\varphi_{m-2}}{\partial q_{m-2}} =
 \sum_{i=2}^m b_{i-1} q_i - \sum_{k=2}^{m-1} \nu_{2k+1} b_{k-1} q_{m} - b_{m-2}q_{m-1}=\sum_{i=2}^{m-2} b_{i-1} q_i.
\end{equation*}
and hence
$$
\varphi_{m-2}(q_1,\ldots,q_{m-2})=\frac{1}{2}b_{m-3}q_{m-2}^2 + \sum_{i=2}^{m-3} b_{i-1} q_{i} q_{m-2}+ \varphi_{m-3}(q_1,\ldots,q_{m-3}).
$$
A recursive argument provides
\begin{equation*}
\frac{\partial\varphi_{m-j}}{\partial q_{m-j}} =\sum_{i=j}^m b_{i-j+1} q_i -\sum_{k=j}^{m-1} \nu_{2k+1}b_{k-j+1} - \sum_{i=m-j+1}^{m-1} b_{i-j+1}q_i =\sum_{i=j}^{m-j} b_{i-j+1} q_i,
\end{equation*}
for all $j=1,\ldots,\lfloor m/2  \rfloor$ and consequently,
\begin{eqnarray*}
\varphi_{m-j}(q_1,\ldots,q_{m-j})&=&\frac{1}{2}b_{m-2j+1}q_{m-j}^2
\\
&& + \sum_{i=j}^{m-j-1} b_i q_i q_{m-j}+ \varphi_{m-j-1}(q_1,\ldots,q_{m-j-1})
\end{eqnarray*}
where for $j=\lfloor m/2 \rfloor$ the function $\varphi_{m-\lfloor m/2 \rfloor -1}$ is constant. Therefore we get a function $V(q)$ with
\begin{align*}
-V(q)&=\nu_1 q_m +
 \sum_{k=1}^{m-1}\nu_{2k+1} \bigg( \frac{1}{2} b_{k+1} q_m^2+ \sum_{i=m-k}^{m-1} b_{i-m+k+1} q_i q_m  \bigg)
 +\frac{1}{3} q_m^3  \\
&+  \sum_{j=1}^{\lfloor m/2 \rfloor }
\big( \frac{1}{2}b_{m-2j+1}q_{m-j}^2 + \sum_{i=j}^{m-j-1} b_i q_i q_{m-j}\big) + \varphi_{m-\lfloor m/2 \rfloor -1},
\end{align*}
such that $\dot p = -\nabla V(q)$. This concludes the proof of
Theorem \ref{thm:hamiltonian}.

\newpage\appendix
\renewcommand{\thesection}{B}
\renewcommand{\theequation}{B.\arabic{equation}}
\setcounter{equation}{0}
\setcounter{thm}{0}
\section{Estimation of $\xi_{1,\lambda_2}$, $\xi_{1,\lambda_3}$ and $\xi_{2,\lambda_1}$}
\label{apendixB}
Let us write the adjoint variational equation
$w^{\prime}(t)=-Df(p(t))^{\ast }w(t)$ as
\begin{equation}
\label{appB-variational}
\left\{
\begin{array}{l}
w'_1(t)=p_1(t)w_3(t)
\\
w'_2(t)=-w_1(t)+w_3(t)
\\
w'_3(t)=-w_2(t)
\end{array}
\right.
\end{equation}
where $w=(w_1,w_2,w_3)$, $p_1(t)=\alpha(-9\tanh(\beta t)+11\tanh^3(\beta t))$, with $\alpha=15\sqrt{11/19^3}$ and $\beta=\sqrt{11/19}/2$, and $p=(p_1,p_2,p_3)$, with $p_2=p'_1$ and $p_3=p''_2$. Writing $s(t)=\tanh(\beta t)$ it follows that
$$
\begin{array}{l}
p_2(t)=\alpha\beta(1-s(t)^2)(-9+33s(t)^2) \quad \text{and} \quad
p_3(t)=\alpha\beta^2(1-s(t)^2)(84s(t)-132s(t)^3).
\end{array}
$$
Our goal is to show that $\xi_{2,\lambda_1}$,  $\xi_{1,\lambda_2}$ and $\xi_{1,\lambda_3}$, as given in (\ref{coefficients}), are different from zero. Note that from the parities of $p_1$, $p_2$, $p_3$, $\psi_3$ and $\varphi_3$, it follows that
\begin{equation}
\label{integrals}
\begin{array}{l}
    {\displaystyle \xi_{2,\lambda_1}=2\int_{0}^{\infty }\psi_3(t)\,dt,}
    \\[0.3cm]
    {\displaystyle \xi_{1,\lambda_2}=2\int_{0}^{\infty }\varphi_3(t)p_3(t)\,dt,}
    \\
    {\displaystyle \xi_{1,\lambda_3}=-4\kappa\int_{0}^{\infty }\varphi_3(t)p_1(t)p_2(t)\,dt.}
\end{array}
\end{equation}
We have already argued in \S\ref{section:nilpotent-3} that each bounded solution $w(t)$ satisfies an orthogonality condition with respect to $f(p(t))$, that is,
$$
p_2(t)w_1(t)+p_3(t)w_2(t)+(c_k^2-p_2(t)-(p_1(t))^2/2)w_3(t)=0.
$$
with $c_k=\sqrt{2}\alpha$. Now,  introducing $v_1(t)=w_3(t)$ and
$v_2(t)=v_1'(t)$, we get  equivalently
\begin{equation}
\label{dae}
A(t)v'(t)=B(t)v(t),
\end{equation}
with $v(t)=(v_1(t),v_2(t))$ and
$$
A(t)=\left(\begin{array}{cc}1&0\\0&p_2(t)\end{array}\right), \qquad
B(t)=\left(\begin{array}{cc}0&1\\(p_1(t))^2/2-c_k^2&p_3(t)\end{array}\right).
$$
Since we have chosen bounded solutions $\varphi(t)$ and $\psi(t)$ of (\ref{appB-variational}) satisfying the initial conditions $\varphi(0)=(0,-1,0)$ and $\psi(0)=(1-c_k^2/p_2(0),0,1)$, they correspond, with respect to the new variables $v_1$ and $v_2$, to solutions $\hat\varphi(t)=(\hat\varphi_1(t),\hat\varphi_1(t))$ and $\hat\psi(t)=(\hat\psi_1(t),\hat\psi_2(t))$ of (\ref{dae}) with initial conditions  $\hat\varphi(0)=(0,1)$ and $\hat\psi(0)=(1,0)$, respectively. Therefore, taking into account that $\kappa\neq 0$, we only have to show that
$$
\int_{0}^{\infty }\hat\psi_1(t)\,dt\neq 0,\quad
\int_{0}^{\infty }\hat\varphi_1(t)p_3(t)\,dt\neq 0\quad\mbox{and}\quad
\int_{0}^{\infty }\hat\varphi_1(t)p_1(t)p_2(t)\,dt\neq 0.
$$
Note that $A(t)$ is singular at  $\hat
t_\pm=\tanh^{-1}(\sqrt{3/11})/\beta\approx \pm 1{.}5529$. Hence
(\ref{dae}) must be treated as a differential algebraic equation
rather than as an ordinary differential equation on the interval
$[0,\infty)$. We will provide approximate values of the integrals
in (\ref{integrals}) using numerical methods on an interval
$[0,t_0]$, with $t_0>\hat t_+$ to be fixed later, and providing
upper bounds for the absolute value of the integrals on
$[t_0,\infty)$. \pagebreak

In order to get appropriate upper bounds on $[t_0,\infty)$  we
write the equation (\ref{dae}) as $v'(t)=Qv(t)+R(t)v(t)$ with
$R(t)=(A(t))^{-1}B(t)-Q$ and
$$
Q=\lim_{t\to\infty}(A(t))^{-1}B(t)=\left(\begin{array}{cc}0&1\\-30/19&-\sqrt{11/19}\end{array}\right).
$$
The eigenvalues of $Q$ are $\lambda_{\pm}=(-\sqrt{209}\pm i \sqrt{2071})38$ and
$$
P=\left(\begin{array}{cc}(-\sqrt{209}-i\sqrt{2071})/60&(-\sqrt{209}+i\sqrt{2071})/60\\1&1\end{array}\right)
$$
is such that $Q=PJP^{-1}$ where $J$ is the complex canonical form of $Q$ with entries $\lambda_+$ and $\lambda_-$ along the diagonal.
From the Lagrange Formula
$$
v(t)=Pe^{J(t-t_0)}P^{-1}v_0+Pe^{Jt}\int_{t_0}^{t}e^{-Js}P^{-1}R(s)v(s)\,ds,
$$
for the solution $v(t)$ of (\ref{dae}) with $v(t_0)=v_0$, we obtain
$$
e^{-at}\|v(t)\|\leq \|P\| e^{-at_0} \|P^{-1}\| \|v_0\|+\|P\|\int_{t_0}^{t}e^{-as}\|P^{-1}\|\|R(s)\|\|v(s)\|\,ds,
$$
where $a$ denotes the real part of $\lambda_{\pm}$. For any $\varepsilon>0$ we can choose $t_0$ such such that $\|R(t)\|<\varepsilon$ for all $t\geq t_0$ and hence, applying the Gronwall Lemma we get
$$
\|v(t)\|\leq \|P\| \|P^{-1}\| \|v_0\| e^{(a+\|P\| \|P^{-1}\| \varepsilon)(t-t_0)}.
$$
Moreover we can assume that $t_0$ is large enough to have $\|R(t)\|$ strictly decrea\-sing on $[t_0,\infty)$ and $\|P\| \|P^{-1}\| \|R(t_0)\|\ll |a|$.  Therefore we can take $\varepsilon=R(t_0)$ and achieve the following upper bounds
$$
\begin{array}{rcl}
\left|\int_{t_0}^{\infty }\hat\psi_1(t)\,dt \right|
& \leq & \int_{t_0}^{\infty }\|\hat\psi(t)\|\,dt  \\
& \leq & \|P\| \|P^{-1}\| \|\hat\psi(t_0)\| \int_{t_0}^{\infty}e^{(a+\|P\|\|P^{-1}\| \|R(t_0)\|)(t-t_0)}\,dt\\
& \leq & L \|\hat\psi(t_0)\|,
\end{array}
$$
with
$$
L=\frac{-\|P\| \|P^{-1}\|}{a+\|P\| \|P^{-1}\| \|R(t_0)\|},
$$
and, analogously,
$$
\begin{array}{lll}
\left|\int_{t_0}^{\infty }\hat\varphi_1(t)p_3(t)\,dt \right|
& \leq & L|p_3(t_0)|\|\hat\varphi(t_0)\|,
\\[1ex]
\left|\int_{t_0}^{\infty }\hat\varphi(t)p_1(t)p_2(t)\,dt \right|
& \leq & L|p_1(t_0)p_2(t_0)|\|\hat\varphi(t_0)\|,
\end{array}
$$
assuming that $t_0$ is chosen such that $|p_3(t)|$ and $|p_1(t)p_2(t)|$ are decreasing on $[t_0,\infty)$.

\begin{table}[t!]
\begin{center}\begin{tabular}{|c|c|c|c|}
\hline
TOL                    & Estimation of              & Estimation of                        & Estimation of
\\
              & $\int_{0}^{20}\hat\psi_1(t)\,dt$ & $\int_{0}^{20}\hat\varphi_1(t)p_3(t)\,dt$  & $\int_{0}^{20}\hat\varphi_1(t)p_1(t)p_2(t)\,dt$
\\
[1ex]
\hline
$    10^{-3}     $                 &  -2.69317886               &  3.38993680                          & 2.14437726\\
$    10^{-4}     $                 &  -2.66952044               &  3.42255576                          & 2.18337593\\
$    10^{-5}     $                 &  -2.65457297               &  3.42416164                          & 2.19168295\\
$    10^{-6}     $                 &  -2.65558754               &  3.48165747                          & 2.27650773\\
$    10^{-7}     $                 &  -2.65577082               &  3.42402265                          & 2.19092815\\
$    10^{-8}     $                 &  -2.65586498               &  3.42412150                          & 2.19145569\\
$    10^{-9}     $                 &  -2.65592518               &  3.42421346                          & 2.19175351\\
$    10^{-10}    $                 &  -2.65594764               &  3.42423492                          & 2.19185340\\
$    10^{-11}    $                 &  -2.65596101               &  3.42424493                          & 2.19188683\\
$    10^{-12}    $                 &  -2.65596369               &  3.42424803                          & 2.19190343\\
$    10^{-13}    $                 &  -2.65596540               &  3.42424892                          & 2.19190641\\
\hline
\end{tabular}\end{center}
\caption{Estimations of the integrals on $[0,20]$ using the code \texttt{ode15s} provided in MATLAB to deal with differential algebraic equations. Note that MATLAB uses two different tolerances, absolute and relative, to compute optimal steps. We take both equal, with the values indicated in the first column.}
\label{numericalresults1}
\end{table}
\begin{table}[t!]
\begin{center}
\begin{tabular}{|c|c|c|c|}
\hline
TOL                    & Estimation of              & Estimation of                        & Estimation of
\\
              & $L \|\hat\psi(20)\|$ & $|p_3(20)|L \|\hat\varphi(20)\|$  & $|p_1(20)p_2(20)|L \|\hat\varphi(20)\|$
\\
\hline
$10^{-13}$ & 4.219110$\times 10^{-2}$& 2.103834$\times 10^{-7}$&3.321829$\times 10^{-7}$
\\
\hline
\end{tabular}
\end{center}
\caption{Upper bounds of the integrals on the interval $[20,\infty]$.}
\label{numericalresults2}
\end{table}

To conclude we have to provide estimations of the integrals on $[0,t_0]$. One can check that, as required, the value $t_0=20$ is such that $\|R(t)\|$, $|p_3(t)|$ and $|p_1(t)p_2(t)|$ are decreasing on $[t_0,\infty)$. Moreover the numerical results show that the values $\|\hat\psi(20)\|$ and $\|\hat\varphi(20)\|$ provide small enough upper bounds for the integrals on $[t_0,\infty)$. As already mentioned (\ref{dae}) must be treated as a differential algebraic equation. The computing environment MATLAB \cite{matlab} provides codes to deal with such kind of equations. For instance, according to \cite{Shampine}, the \texttt{ode15s} code is an appropriate one to solve DAEs. We have used such method working with different tolerances to get the results shown in Table \ref{numericalresults1}. Once the numerical solution is obtained, the integrals were approximated using the trapezoidal rule with the nodes which were generated by the numerical algorithm. With the highest tolerance value, the maximum step sizes are $0{.}01355215$ and $0{.}01323550$ for the solutions $\hat\psi$ and $\hat\varphi$, respectively. On the other hand one can get accurate estimations of the upper bounds for the integrals on $[t_0,\infty]$. The results taking the approximates values of the solutions at $t=t_0$ obtained with the highest tolerance are shown in Table \ref{numericalresults2}.

\begin{rem}
{\rm All the estimations of upper bounds are obtained using the maximum norm. It can be easily checked that $\|P\|=2$, $\|P^{-1}\|=\sqrt{30/109}+30/\sqrt{2071}$ and $\|R(t)\|=(|p_1(t)^2/2-c_k^2+30/19|+|p_3(t)+\sqrt{11/19}|)/|p_2(t)|$.}
\end{rem}

\begin{rem}
{\rm The code \texttt{ode23t} included in MATLAB also solves differential algebraic equations. Even \texttt{ode45} can be used to solve our particular equation. The results does not change significatively, namely, the differences are below $10^{-4}$. We have also used our own algorithms (a Taylor method of order 25 for the variational equation with small fixed step and projection on the stable manifold). Again the differences are below $10^{-4}$.}
\end{rem}

\newpage\appendix
\renewcommand{\thesection}{C}
\renewcommand{\theequation}{C.\arabic{equation}}
\setcounter{equation}{0} \setcounter{thm}{0}
\section{Proof of Lemma~\ref{lem:eq-bif}}
\label{apendixC}

As it stated in Proposition~\ref{lem:conexion} the variational
equation $z'=Df(p(t))z$ has exponential dichotomy in
$[t_0,\infty)$ and $(-\infty,t_0]$. Let
$$
\mathscr{P}_+(t_0)=X(t_0)P_+X^{-1}(t_0) \quad \text{and} \quad
I-\mathscr{P}_-(t_0)=I-X(t_0)P_-X^{-1}(t_0)
$$ be the corresponding
projection matrix on the stable space
$E^s_{t_0}=T_{p(t_0)}W^s(p_+)$ and instable space
$E^u_{t_0}=T_{p(t_0)}W^u(p_-)$, respectively.

Before to give the proof of Lemma~\ref{lem:eq-bif} we need the
following preliminar result.

\begin{lem} \label{lem2:2.1}
Let $b \in C^0_b([t_0,\infty),\bb^n)$. Then,
 $z^+(t)$ is a positively bounded solution of
 $$z'=Df(p(t))z + b(t)$$ if and only if
\begin{equation} \label{acotada+}
\begin{alignedat}{1}
  z^+(t) &= X(t)X^{-1}(t_0)\mathscr{P}_+(t_0)z^+(t_0)  \\
 &+\int_{t_0}^{t} X(t)X^{-1}(s)
\mathscr{P}_{+}(s) b(s)\,ds - \int_{t}^{\infty}
X(t)X^{-1}(s)(I-\mathscr{P}_{+}(s))b(s)\,ds.
\end{alignedat}
\end{equation}
\end{lem}
\begin{proof}
Since $X(t)$ is the fundamental matrix of the lineal homogeneous
equation $z'=Df(p(t))z$, the solution of the complete linear
equation $z'=Df(p(t))z + b(t)$ are
\begin{align*}
   z(t)&=X(t)X^{-1}(t_0)z(t_0) + X(t)\int_{t_0}^{t} X^{-1}(s)b(s)\, ds. %\\
 %&=\Phi(t,t_0)z(t_0)+\int_{t_0}^t \Phi(t,s) b(s)\, ds.
\end{align*}
By means of the projection $P_+$ in the exponential dichotomy of
the homogeneous equation in $[t_0,\infty)$, this solution can be
written as
\begin{equation} \label{eq2:2.4}
\begin{alignedat}{1}
  z(t) &= X(t)P_+X^{-1}(t_0)z(t_0) + X(t)(I-P_+)X^{-1}(t_0)z(t_0) \\
 &+X(t)\int_{t_0}^{t} P_{+}X^{-1}(s)b(s)\,ds + X(t)\int_{t_0}^{t} (I-P_{+})X^{-1}(s)b(s)\,ds.
\end{alignedat}
\end{equation}
On the other hand, according to the exponential dichotomy
%\begin{equation*}
$\|X(t)P_+X^{-1}(s)\|\leq Ke^{-\alpha(t-s)}$ for %\qquad \text{for} \
$t\geq s \geq t_0$,
%\end{equation*}
it follows that
\begin{equation*}
\begin{alignedat}{2}
 |X(t)P_+X^{-1}(t_0)z(t_0)| &\leq K e^{-\alpha (t-t_0)} |z(t_0)| \quad &\text{for} \ t\geq t_0, \\
 |X(t)\int_{t_0}^{t} P_{+}X^{-1}(s)b(s)\,ds| & \leq \int_{t_0}^{t} K e^{-\alpha (t-s)} |b(s)|\,ds \quad &\text{for} \ t\geq t_0,
\end{alignedat}
\end{equation*}
and thus, the first and third term of \eqref{eq2:2.4} are bounded
for $t\geq t_0$.

If we assume that $z(t)$ is bounded, necessarily then the sum
\begin{equation} \label{cota}
\begin{alignedat}{1}
X(t)&(I-P_+)X^{-1}(t_0)z(t_0)+
X(t)\int_{t_0}^{t} (I-P_{+})X^{-1}(s)b(s)\,ds \\
&= X(t) [(I-P_+)X^{-1}(t_0)z(t_0)+ \int_{t_0}^{t}
(I-P_{+})X^{-1}(s)b(s)\,ds]
\end{alignedat}
\end{equation}
is also bounded. However, the  exponential dichotomy
$\|X(t)(I-P_+)X^{-1}(s)\|\leq L e^{-\beta(s-t)}$ for $s\geq t \geq
t_0,$ implies that \begin{align*} |X(t_0)(I-P_+)X^{-1}(t_0)z(t_0)|
&= |X(t_0)(I-P_+)X^{-1}(t)X(t)(I-P_+)X^{-1}(t_0)z(t_0)| \\
&\leq  L e^{-\beta (t-t_0)} |X(t)(I-P_+)X^{-1}(t_0)z(t_0)|.
\end{align*}
Hence,
\begin{equation*}
 |X(t)(I-P_+)X^{-1}(t_0)z(t_0)|\geq  |X(t_0)(I-P_+)X^{-1}(t_0)z(t_0)| L^{-1} e^{\beta (t-t_0)}
\end{equation*}
is not bounded and since  $|X(t)(I-P_+)X^{-1}(t_0)z(t_0)| \leq
\|X(t)\||(I-P_+)X^{-1}(t_0)z(t_0)|$, it follows that the matrix
$X(t)$ is not bounded. Therefore, from~\eqref{cota} we get that
the solution $z(t)$ only can be bounded if it holds
 \begin{equation*}
 (I-P_+)X^{-1}(t_0)z(t_0)+\int_{t_0}^{\infty} (I-P_+)X^{-1}(s)b(s)\, ds=0.
 \end{equation*}
Consequently, from~\eqref{eq2:2.4} we obtain that if $z^+(t)$ is a
bounded solution of~$z'=Df(p(t))z+b(t)$ then
\begin{equation*}
\begin{alignedat}{1}
  z^+(t) &= X(t)X^{-1}(t_0)\mathscr{P}_+(t_0)z^+(t_0)  \\
 &+\int_{t_0}^{t} X(t)X^{-1}(s)
\mathscr{P}_{+}(s) b(s)\,ds - \int_{t}^{\infty}
X(t)X^{-1}(t_0)(I-\mathscr{P}_{+}(s))b(s)\,ds.
\end{alignedat}
\end{equation*}
Conversely, to verify that $z^+(t)$ given in~\eqref{acotada+} is a
bounded solution of $z'=Df(p(t))z+b(t)$ it suffices to see that
\begin{equation*}
L^+b(t)=\int_{t_0}^{t} X(t)X^{-1}(s) \mathscr{P}_{+}(s) b(s)\,ds -
\int_{t}^{\infty} X(t)X^{-1}(t_0)(I-\mathscr{P}_{+}(s))b(s)\,ds
\end{equation*}
is a particular bounded solution of the above complete linear
equation. Indeed,
\begin{equation*}
\begin{alignedat}{2}
 &\frac{d}{dt} L^+b(t) & &= Df(p(t))X(t)\big[ \int_{t_0}^{t} P_{+}X^{-1}(s)b(s)\,ds - \int_{t}^{\infty} (I-P_{+})X^{-1}(s)b(s)\,ds \big] \\
 & & &+ X(t)\big[ P_+X^{-1}(t)b(t)+ (I-P_+)X^{-1}(t)b(t) \big] =Df(p(t)) L^+b(t)+ b(t), \\
&|L^+b(t)| & &\leq \int_{t_0}^{t} K e^{-\alpha (t-s)} |b(s)|\, ds
+ \int_{t}^{\infty} L e^{-\beta (s-t)} |b(s)|\, ds \leq \tilde K +
\tilde L,
\end{alignedat}
\end{equation*}
for all  $t\geq t_0$, where the constant $\tilde K$ and $\tilde L$
not depend on $t$.
\end{proof}

In the same way, a similar result to the negative bounded
solutions of the complete linear equation is followed.
\begin{lem} \label{lem 2.3}
Let $b \in C^0_a((-\infty,t_0],\bb^n)$. Then,
 $z^-(t)$ is a negative bounded solution of $$z'=Df(p(t))z + b(t)$$ if and only if
\begin{equation} \label{acotada-}
\begin{alignedat}{1}
  z^-(t) &= X(t)X^{-1}(t_0)(I-\mathscr{P}_-(t_0))z^-(t_0)  \\
 &+\int_{-\infty}^{t}X(t)X^{-1}(s)
\mathscr{P}_-(s) b(s)\,ds - \int_{t}^{t_0}
X(t)X^{-1}(s)(I-\mathscr{P}_{-}(s))b(s)\,ds.
\end{alignedat}
\end{equation}
\end{lem}

\begin{rem} \label{nota:solucion}
Notice that the functions given in~\eqref{acotada+}
and~\eqref{acotada-} are both solutions of the equation
$z'=Df(p(t))z+b(t)$ for any continuous function $b(t)$ on
$[t_0,\infty)$ and $(-\infty,t_0]$, respectively. On the other
hand, to prove that both solutions are bounded solutions of this
equation we need to use that $b(t)$ is also bounded.
\end{rem}

Recall that $E^s_{t_0}=T_{p(t_0}W^s(p_+)$ and
$E^u_{t_0}=T_{p(t_0)}W^u(p_-)$.
%Since the (homo)heteroclinic orbit $\gamma$ is non degenerate,
%according to Remark~\ref{rem:d}, the maximum number $d$ of bounded
%solution of the adjoint variational equation $w'=-Df(p(t))^*w$
%coincide with $d=s_--s_++1$ where $s_\pm$ are the stable indexes
%of $p_+$ and $p_-$ respectively.
Notice that since the (homo)heteroclinic orbit $\gamma$ is non
degenerate, $E_{t_0}=E^s_{t_0}\cap E^u_{t_0}$ is one-dimensional.
In fact, this space is generated by the vector
$p'(t_0)=f(p(t_0))$. Moreover, according to Remark~\ref{rem:d} the
dimension of
$$
E^*_{t_0}=E^{s*}_{t_0}\cap E^{u*}_{t_0}=[E^s_{t_0} +
E^u_{t_0}]^\bot
$$
is $d=s_--s_++1$ where $s_\pm$ are the stable indexes of $p_+$ and
$p_-$ respectively.

Let us define $W^{\pm}_{t_0}$ the ortogonal complement to
$E_{t_0}$ in the tangent space of the stable and unstable
manifolds, respectively, that is $E^s_{t_0}=E_{t_0} \bot
W^+_{t_0}$ and $E^u_{t_0}=E_{t_0} \bot W^-_{t_0}$. Then
\begin{equation} \label{eq:descomposicion}
\mathbb{R}^n= \mathrm{span} \{ f(p(t_0)) \} \oplus \, W^+_{t_0}
\oplus W^-_{t_0} \oplus E^*_{t_0}.
\end{equation}
Finally, we take the transversal section to $\gamma$ at $p(t_0)$
given by
$$
\Sigma_{t_0}=p(t_0)+\{f(p(t_0))\}^\bot=p(t_0)+[W^+_{t_0} \oplus
W^-_{t_0} \oplus E^*_{t_0}]
$$
and a base $\{ w_i: i=1\dots d\}$ de $E^*_{t_0}$. Notice that
$$
w_i(s)=X^{-1}(s)^*X(t_0)^*w_i \quad \text{for $i=1,\dots, d$}
$$
are  bounded linearly independent solutions of the adjoint
variational equation. % $w'=-Df(p(t))^*w$.

%\begin{lem}  \label{lem:eq-bif}
% Existe $\delta>0$ tal que para todo $\lambda\in\bb^k$ con $|\lambda|<\delta$ se verifica
%\begin{enumerate}
% \item Existe único par de soluciones  $p^+_\lambda(t)$  y $p^-_\lambda(t)$ de~(\ref{eq2:2.1}) parametrizando órbitas sobre la variedad estable $W^s(p_+(\lambda))$ e inestable $W^u(p_-(\lambda))$, respectivamente, y    tales que
%$p^{\pm}_{\lambda}(t_0) \in \Sigma_{t_0}$ y
%\begin{equation*}
% \xi^{\infty}(\lambda)=p_{\lambda}^-(t_0)-p_{\lambda}^{+}(t_0)\in E^*_{t_0}.
%\end{equation*}
%Expresadas las soluciones de la forma
%$p_{\lambda}^{\pm}(t)=p(t)+z^{\pm}_\lambda(t)$, entonces
%$z^{\pm}_\lambda(\cdot)$ son, respectivamente, soluciones positiva
%y negativamente acota\-das  de la ecuación~(\ref{eq2:2.2}).
%Dependen  regularmente de $\lambda$ y las funciones $z^{\pm}_0$
%son  idénticamente nulas.
%\item Para $\varepsilon>0$ suficientemente pequeño existe una órbita homoclínica $p_{\lambda}(t)$ tal que $|p_{\lambda}(t_0)-p(t_0)|<\varepsilon$ si y sólo si
%$\xi^{\infty}(\lambda)=0$, es decir, las componentes
%$\xi_i^\infty(\lambda)$  de $\xi^{\infty}(\lambda)$ en la base $\{
%w_i: i=1\dots d\}$ de $E^*_{t_0}$ verifica
%\begin{equation*}
%\xi_i^{\infty}(\lambda)\equiv \int_{-\infty}^{t_0}< w_i(s),
%b(s,z^-_\lambda(s),\lambda)>\,ds+\int_{t_0}^{\infty}< w_i(s),
%b(s,z^+_\lambda(s),\lambda)>\,ds =0.
%\end{equation*}
%\end{enumerate}
%\end{lem}
\begin{proof}[Proof of Lemma~\ref{lem:eq-bif}]
%%$b_{\lambda}$ el operador sobre el espacio de Banach $C^0_a(J,\bb^n)$ definido  por $b_{\lambda} z(\cdot)=b(\cdot, z(\cdot), \lambda)$ donde $J$ será $J_+=[t_0,\infty)$ ó $J_-=(-\infty,t_0]$.
The solutions $p_{\lambda}(t)$ of~\eqref{*1} can be written as
$p_{\lambda}(t)=p(t)+z_{\lambda}(t)$ where $z_{\lambda}(t)$ is a
solution of~\eqref{*2}. Let $Y_{t_0}=W^+_{t_0}\oplus W^-_{t_0}
\oplus E^*_{t_0}$. Assuming $p_{\lambda}(t_0)\in \Sigma_{t_0}$
then $z_{\lambda}(t_0)\in Y_{t_0}$.

In order to get that $p_{\lambda}(t)$ parametrizes an orbit in
$W^s(p_+(\lambda))$ (resp.  $W^u(p_-(\lambda))$), the function
$z_{\lambda}(t)$ have to be a positively (resp.~negatively)
bounded solution of~\eqref{*2}. Now if we assume that $z^\pm(t)$
are a pair of positively and negatively bounded solutions
of~\eqref{*2} respectively, then $b(\cdot,z^\pm(\cdot),\lambda)
\in C_b^0(J_\pm,\bb^n)$ where $J_+=[t_0,\infty)$ and
$J_-=(-\infty,t_0]$. Thus, according to Lemma~\ref{lem2:2.1} and
Lemma~\ref{lem 2.3} it must be met that
\begin{equation} \label{2.8}
\begin{alignedat}{1}
  z^+(t) &= X(t)X^{-1}(t_0)\mathscr{P}_+(t_0)z^+(t_0) \\
 &+\int_{t_0}^{t} X(t)X^{-1}(s)
\mathscr{P}_{+}(s) b(s,z^\pm(s),\lambda)\,ds  - \int_{t}^{\infty}
X(t)X^{-1}(s)(I-\mathscr{P}_{+}(s))b(s,z^\pm(s),\lambda)\,ds
\end{alignedat}
\end{equation}
and
\begin{equation} \label{2.9}
\begin{alignedat}{1}
  z^-(t) &= X(t)X^{-1}(t_0)(I-\mathscr{P}_-(t_0))z^-(t_0) \\ &+\int_{-\infty}^{t} X(t)X^{-1}(s)
\mathscr{P}_-(s) b(s,z^\pm(s),\lambda)\,ds   - \int_{t}^{t_0}
X(t)X^{-1}(s)(I-\mathscr{P}_{-}(s))b(s,z^\pm(s),\lambda)\,ds.
\end{alignedat}
\end{equation}

Conversely, according to Remark~\ref{nota:solucion}, the solutions
$z^+(t)$ and $z^-(t)$ of the integral equations~\eqref{2.8}
and~\eqref{2.9} are both solutions of~\eqref{*2}, but not
necessarily bounded. The existence of positively and negatively
bounded solutions of~\eqref{2.8} and~\eqref{2.9} will be proved as
an application of Implicit Function Theorem.

Ler $\eta^+ =\mathscr{P}_+(t_0)z^+(t_0) \in W^+_{t_0}$ and
$\eta^-=(I-\mathscr{P}_-(t_0))z^-(t_0) \in W^-_{t_0}$.
Equations~\eqref{2.8} and \eqref{2.9} can be written in the form
\begin{equation} \label{eq:H}
z^\pm=\mathscr{H}^\pm(z^\pm,\eta^\pm,\lambda)
\end{equation}
where $\mathscr{H}^\pm : C^0_b(J_\pm, \mathbb{R}^n)\times
W^\pm_{t_0} \times \mathbb{R}^k \to C^0_b(J_\pm, \mathbb{R}^n)$.
In order to apply the Implicit Function Theorem to the equation
$z-\mathscr{H}^\pm(z^\pm,\eta^\pm,\lambda)=0$ notice first that
$\mathscr{H}^\pm(0,0,0)=0$. On the other hand,
$D_z\mathscr{H}^\pm(z^\pm,\eta^\pm,\lambda): C^0_b(J_\pm, \bb^n)
\to C^0_b(J_\pm, \bb^n)$ is the null function for $z^\pm=\eta^\pm=
\lambda=0$. Indeed, for any $h \in C^0_b(J_+, \bb^n)$ it holds
that
\begin{align*}
  D_z\mathscr{H}^+(z,\eta,\lambda) h(t) &=
\int_{t_0}^{t}\Phi(t,s)\mathscr{P}_+(s)D_zb(s,z(s),\lambda)h(s)\,ds
\\ &- \int_{t}^{\infty}
\Phi(t,s)(I-\mathscr{P}_+(s))D_zb(s,z(s),\lambda) h(s)\,ds.
\end{align*}
Since $D_zb(s,0,0)=0$ then $D_z\mathscr{H}^+(0,0,0)=0$. Similarly
it follows that $Dz\mathscr{H}^-(0,0,0)=0$. Therefore, there
exists $\delta_\pm
>0$ such that for every $\eta^\pm\in W^\pm_{t_0}$ and $\lambda\in
\bb^k$ with $|\eta^\pm|$, $|\lambda|<\delta_\pm$ there is a unique
$z^\pm(\eta^\pm,\lambda) \in C^0_b(J_\pm,\bb^n)$ so that
$z^\pm(\eta^\pm,\lambda)=\mathscr{H}^\pm(z^\pm(\eta^\pm,\lambda),\eta^\pm,\lambda)$
and $z^\pm(0,0)=0$.

Now, consider the condition $z^-(\lambda,\eta^-)(t_0)-z^+(\lambda,
\eta^+)(t_0)\in E^*_{t_0}$. Since $Y_{t_0}=W^+_{t_0}\oplus
W^-_{t_0} \oplus E^*_{t_0}$, we can write
\begin{equation*} \label{eq2:descomposicion}
    z^\pm(\eta^\pm,\lambda)(t_0)=\eta^\pm + w^\mp(\eta^\pm,\lambda)+\varrho^\pm(\eta^\pm,\lambda)
\end{equation*}
where $w^\mp(\eta^\pm,\lambda) \in W^\mp_{t_0}$ and
$\varrho^\pm(\eta^\pm,\lambda) \in E^*_{t_0}$. From~\eqref{2.8}
and ~\eqref{2.9}, and having into account that
$X(t)X^{-1}(s)\mathscr{P}_\pm(s)=\mathscr{P}_\pm(t)X(t)X^{-1}(s)$,
it follows that
\begin{equation} \label{eq:romanpaladin}
\begin{alignedat}{1}
z^-(\eta^-,\lambda)(t_0) &=\eta^- + \mathscr{P}_-(t_0) \int_{-\infty}^{t_0} X(t_0)X^{-1}(s)b(s,z^-(\eta^-,\lambda)(s),\lambda)\, ds, \\
z^+(\eta^+,\lambda)(t_0) &= \eta^+ -(I-\mathscr{P}_+(t_0))
\int_{t_0}^\infty
X(t_0)X^{-1}(s)b(s,z^+(\eta^+,\lambda)(s),\lambda)\, ds.
\end{alignedat}
\end{equation}
Thus
\begin{equation} \label{eq:w+}
\begin{alignedat}{1}
w^+(\eta^-,\lambda)+\varrho^-(\eta^-,\lambda)&= \mathscr{P}_-(t_0) \int_{-\infty}^{t_0} X(t_0)X^{-1}(s)b(s,z^-(\eta^-,\lambda)(s),\lambda)\, ds, \\
w^-(\eta^+,\lambda)+\varrho^+(\eta^+,\lambda) &=
-(I-\mathscr{P}_+(t_0)) \int_{t_0}^\infty
X(t_0)X^{-1}(s)b(s,z^+(\eta^+,\lambda)(s),\lambda)\, ds.
\end{alignedat}
\end{equation}

Recall that $z^\pm(0,0)=0$ and hence $w^\pm(0,0)=\rho^\pm(0,0)=0$.
On the other hand, applying that $D_zb(t,0,0)=0$ in~\eqref{eq:w+}
we get that
\begin{equation} \label{2.12}
\begin{alignedat}{1}
D_{\eta^\mp}\varrho^\mp(0,0)&=D_{\lambda}\varrho^\mp(0,0)=0, \\
D_{\eta^\mp}w^\pm(0,0)&=D_{\lambda}w^\pm(0,0)=0.
\end{alignedat}
 \end{equation}

The condition $z^-(\lambda,\eta^-)(t_0)-z^+(\lambda,
\eta^+)(t_0)\in E^*_{t_0}$ is equivalent to the system of two
equations
%\begin{equation} \label{2.11}
$\eta^\pm-w^\pm(\eta^\mp,\lambda)= 0$
%\end{equation}
which we write in the form $F(\eta,\lambda)=0$ where
$\eta=(\eta^+,\eta^-)\in W^+_{t_0} \times W^-_{t_0}$. We have that
$F(0,0)=0$ and according to~\eqref{2.12} it follows that
$D_{\eta}F(0,0)=I$. Hence, applying again the Implicit Function
Theorem we get $\eta$ as a function of $\lambda$. We conclude that
there exists $\delta>0$ ($\delta< \delta_\pm$) such that for any
$\lambda \in \bb^k$ with $|\lambda|<\delta$ there is a unique
$\eta^\pm(\lambda)\in W^\pm_{t_0}$ so that
$\eta^\pm(\lambda)-w^\pm(\eta^\mp(\lambda),\lambda)= 0$ and
$\eta^\pm(0)=0$.

Now consider the functions
$z^\pm_\lambda(t)=z^\pm(\eta^\pm(\lambda),\lambda)(t)$. The
uniqueness, boundedness and regularity respect to $\lambda$ of
$z^\pm_\lambda(t)$, and thus of
$p_\lambda^\pm(t)=p(t)+z^\pm_\lambda(t)$, are following from the
Implicit Function Theorem. Also we have that $z_0^\pm=0$ and since
for $|\lambda|$ small enough $z^\pm_{\lambda}$ is close to
$z^\pm_0=0$ we get that $\sup_{t\in J_\pm}
|p_{\lambda}^\pm(t)-p(t)|$ is arbitrarily small. This means that
the orbits parameterize by  $p_{\lambda}^\pm(t)$ are close to the
orbit $\gamma$ which is  parameterized by  $p(t)$. So, together
the hyperbolicity of the equilibria $p_\pm(\lambda)$ we get that
$\lim_{t\to \pm \infty} p_\lambda^\pm(t)=p_\pm(\lambda)$. Thus,
for each  $\lambda$ with $|\lambda|<\delta$, the solutions
$p_\lambda^\pm(t)$ parameterizes, respectively, orbits in the
stable manifold $W^{s}(p_+(\lambda))$ and in the unstable manifold
$W^{u}(p_-(\lambda))$ such that $p^\pm_{\lambda}(t_0)\in
p(t_0)+Y_{t_0} = \Sigma_{t_0}$ and
$\xi^{\infty}(\lambda)=p_\lambda^-(t_0)-p_\lambda^+(t_0)=
z_\lambda^-(t_0)-z^+_\lambda(t_0) \in E^*_{t_0}$. This proves the
first item of the lemma.

\begin{figure}
\begin{center}
\scalebox{1}{\input{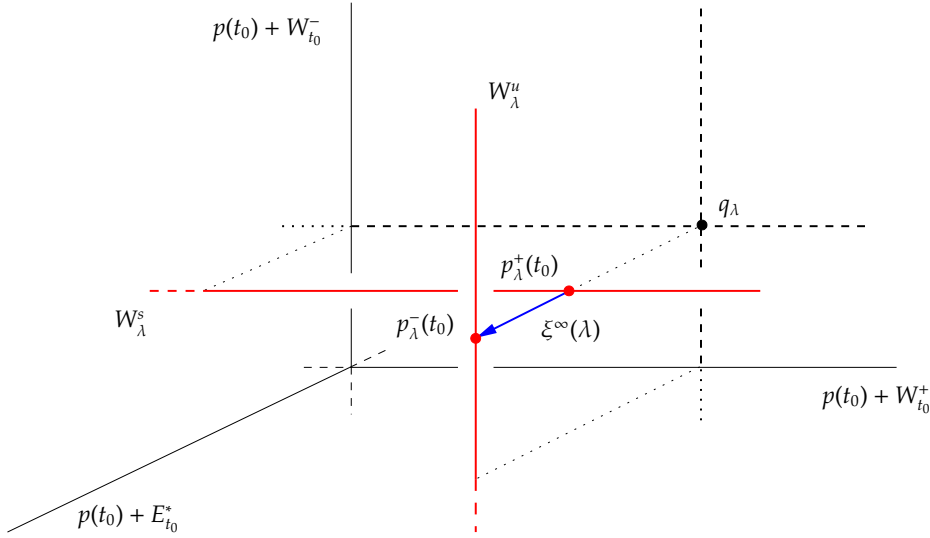}} \caption{The figure shows the
transversal section $\Sigma_{t_0}=p(t_0)+[W^+_{t_0}\oplus
W^-_{t_0}\oplus E^*_{t_0}]$ at the point $p(t_0)$ to a non
degenerate (homo)heteroclinic orbit $\gamma=\{p(t): t\in
\mathbb{R}\}$. The curves $W^s_{\lambda}$ and $W^u_{\lambda}$ are,
respectively, the intersections
$W^s(p_+(\lambda))\cap\Sigma_{t_0}$ and
$W^u(p_-(\lambda))\cap\Sigma_{t_0}$.  A priori the curve
$W^s_{\lambda}$ does not met $W^u_{\lambda}$. However, the
projection along the direction $E^*_{t_0}$ of both curves on
$p(t_0)+[W^-_{t_0} \oplus W^+_{t_0}]$ have a unique transversal
intersection point $q_{\lambda}$. Now, $q_\lambda + E^*_{t_0}$
mets, respectively, $W^s_{\lambda}$ and $W^u_{\lambda}$ at
$p_\lambda^+(t_0)$ and $p_{\lambda}^-(t_0)$. This two points
define the vector $\xi^\infty(\lambda)$. The persistence of the
connection holds if $\xi^\infty(\lambda)=0$ which provides a set
of $d=\dim E^*_{t_0}$ conditions.} \label{fig:E'}
\end{center}
\end{figure}

In order to prove the second item notice that if
$\xi^\infty(\lambda)=0$ then a (homo)heteroclinic orbit of
~\eqref{*1} is given by
\begin{equation*}
p_{\lambda}(t)=\left\{
\begin{alignedat}{1}
 p_\lambda^-(t) & \quad  \text{para} \ t\leq t_0, \\
p_\lambda^+(t) &  \quad \text{para} \ t\geq t_0.
\end{alignedat}
\right.
\end{equation*}
where $p^\pm_\lambda(t)$ are the solutions in the first item. On
the other hand, if $p_\lambda(t)$ is a solution parametrising a
(homo)heteroclinic connection such that
$p_\lambda(t_0)\in\Sigma_{t_0}$ with $|\lambda|$  and
$|p_\lambda(t_0)-p(t_0)|$ small enough, its restriction to the
intervals $J_-=(-\infty,t_0]$ and $J_+=[t_0,\infty)$ define a pair
of solutions  $p_\lambda^\pm(t)$ in the assumption of the first
item. That is, $p^\pm_\lambda(t_0)\in\Sigma_{t_0}$ and
$p_\lambda^-(t_0)-p_\lambda^+(t_0)=0\in E^*_{t_0}$. This makes
obvious the reciprocal implication.

To conclude the proof of the second item notice that
\begin{equation*} \xi^{\infty}(\lambda)=
p^{-}_{\lambda}(t_0)-p^{+}_{\lambda}(t_0) = z^-_\lambda(t_0)
-z^+_\lambda(t_0) \in E^*_{t_0}.
\end{equation*}
Since $\{ w_i: i=1\dots d\}$ is a base of
$E^*_{t_0}=E^{s*}_{t_0}\cap E^{u*}_{t_0}=
[E^s_{t_0}+E^u_{t_0}]^\bot$ then
\begin{equation*}
  \xi^\infty(\lambda)=\sum_{i=1}^d
  <w_i,\xi^\infty(\lambda)>w_i.
\end{equation*}
Form~\eqref{eq:romanpaladin} and having into account that
$<w_i,\eta^\pm>=0$, it follows that
\begin{equation} \label{eq:cita}
\begin{alignedat}{1}
 \xi^\infty_i(\lambda)&=<w_i,\xi^{\infty}(\lambda)>
= <w_i,\mathscr{P}_-(t_0)\int_{-\infty}^{t_0}
X(t_0)X^{-1}(s) b(s, z^-_\lambda(s), \lambda) \, ds > \\
&+ <w_i, (I-\mathscr{P}_+(t_0))\int^{\infty}_{t_0} X(t_0)X^{-1}(s) b(s, z^+_\lambda(s), \lambda) \, ds> \\
&=  \int_{-\infty}^{t_0}
<[\mathscr{P}_-(t_0)X(t_0)X^{-1}(s)]^*w_i,
 b(s, z^-_\lambda(s), \lambda) \, ds > \\
&+ \int^{\infty}_{t_0}<[(I-\mathscr{P}_+(t_0))
X(t_0)X^{-1}(s)]^*w_i,  b(s, z^+_\lambda(s), \lambda) \, ds>.
\end{alignedat}
\end{equation}
Thus, since $\mathscr{P}_-(t_0)^*: \bb^n\to E^{u*}_{t_0}$ and
$I-\mathscr{P}_+(t_0)^*: \bb^n\to E^{s*}_{t_0}$ we get that
\begin{align*}
 [\mathscr{P}_-(t_0)X(t_0)X^{-1}(s)]^*w_i &=X^{-1}(s)^*X(t_0)^*\mathscr{P}_-(t_0)^*w_i= w_i(s) \\
[(I-\mathscr{P}_+(t_0))X(t_0)X^{-1}(s)]^*w_i
&=X^{-1}(s)^*X(t_0)^*(I-\mathscr{P}_+(t_0)^*)w_i= w_i(s).
\end{align*}
Substituting  in~\eqref{eq:cita} we obtain that
\begin{equation*}
\xi_i^{\infty}(\lambda)\equiv \int_{-\infty}^{t_0}< w_i(s),
b(s,z^-_\lambda(s),\lambda)>\,ds+\int_{t_0}^{\infty}< w_i(s),
b(s,z^+_\lambda(s),\lambda)>\,ds =0.
\end{equation*}
This concludes the second item and proves Lemma~\ref{lem:eq-bif}.
\end{proof}

\newpage
\def\cprime{$'$}

\end{document}